\documentclass[reqno,A4paper]{amsart}

\usepackage{amssymb}
\usepackage{graphicx}
\usepackage{amsmath}
\usepackage{amsthm}
\usepackage{amsfonts}
\usepackage{pgf,tikz}
\usepackage{mathrsfs}
\usetikzlibrary{arrows}
\usepackage[all]{xy}

\usepackage{amsmath}
\usepackage{amssymb}
\usepackage{amsthm}
\usepackage{enumerate}
\usepackage{mathrsfs} 
\usepackage{eqlist}
\usepackage{array}

\theoremstyle{plain}
\newtheorem{theorem}{Theorem}[section]
\newtheorem{prop}[theorem]{Proposition}
\newtheorem{lemma}{Lemma}[section]
\newtheorem{corol}{Corollary}[theorem]

\theoremstyle{definition}
\newtheorem{definition}{Definition}
\newtheorem{remark}{\textup{Remark}} 
\newtheorem{example}{\textit{Example}} 

\numberwithin{equation}{section}

\begin{document}

\title[FIDL-modules: representation and duality]
{Modules with fusion and implication based over distributive lattices: representation and duality}
\author[Ismael Calomino \and William J. Zuluaga Botero]
{Ismael Calomino* \and William J. Zuluaga Botero**}

\newcommand{\acr}{\newline\indent}

\address{\llap{*\,}CIC and \acr
                   Departamento de Matem\'{a}ticas \acr
                   Facultad de Ciencias Exactas \acr
                   Universidad Nacional del Centro \acr
                   Tandil, ARGENTINA}
\email{calomino@exa.unicen.edu.ar}

\address{\llap{**\,}Laboratoire J. A. Dieudonn\'{e} \acr
                    Universit\'{e} C\^ote d'Azur \acr 
                    Nice, FRANCE \acr
                    and \acr
                    Departamento de Matem\'{a}ticas \acr
                    Facultad de Ciencias Exactas \acr
                    Universidad Nacional del Centro \acr             
                    Tandil, ARGENTINA}                   
\email{wizubo@gmail.com}


\thanks{This work was supported by the CONICET under Grant PIP 112-201501-00412}

\subjclass[2010]{Primary 06D50, 06D05; Secondary 06D75}
\keywords{Distributive lattice, module, Priestley-like duality}

\begin{abstract}
In this paper we study the class of {\it{modules with fusion and implication based over distributive lattices}}, or {\it{FIDL-modules}}, for short. We introduce the concepts of FIDL-subalgebra and FIDL-congruence as well as the notions of simple and subdirectly irreducible FIDL-modules. We give a bi-sorted Priestley-like duality for FIDL-modules and moreover, as an application of such a duality, we provide a topological bi-spaced description of the FIDL-congruences. This result will allows us to characterize the simple and subdirectly irreducible FIDL-modules.
\end{abstract}

\maketitle

\section{Introduction} \label{Introduction}

Bounded distributive lattices with additional operators occur often as algebraic models of non-classical logics. This is the case of Boolean algebras which are the algebraic semantics of classical logic, Heyting algebras which model intuitionistic logic, BL-algebras which correspond to algebraic semantics of basic propositional logic (\cite{Hohle}), MTL-algebras which are the algebraic semantics of the basic fuzzy logic of left-continuous t-norms (\cite{Esteva-Godo,Cabrer-Celani}), Modal algebras which model propositional modal logics (\cite{Chagrov,Venema}), to name a few. In all these cases, the binary operations $\vee$ and $\wedge$ model logical disjunction and conjunction and the additional operations are usually interpretations of other logical connectives such as the modal necessity ($\Box$) or modal possibility ($\Diamond$), or various types of implication. All these operations has as a common property: the preservation of some part of the lattice structure, for example, the necessity modal operator satisfies the conditions $\Box1=1$ and $\Box(x\wedge y)=\Box (x) \wedge \Box (y) $, or the possibility modal operator $\Diamond 0 =0$ and $\Diamond (x \vee y)= \Diamond (x) \vee \Diamond (y)$. 

In some sense, the aforementioned may suggest that these ideas can be treated as a more general phenomenon which can be studied by employing tools of universal algebra. Some papers in which this approach is used are \cite{Goldblatt} and \cite{Stokkermans}. Nevertheless, in an independent way, a more concrete treatment of the preservation of the lattice structure by two additional connectives in a distributive lattice leads to the introduction of the class of distributive lattices with fusion and implication in \cite{Celani1}, which encompasses all the algebraic structures mentioned before. 

The aim of this paper is to introduce the class of {\it{modules with fusion and implication based over distributive lattices}}, for short, {\it{FIDL-modules}}. The FIDL-modules generalize both distributive lattices with fusion and implication and modal distributive lattices, giving a different approach to study these structures. A bi-sorted Priestley-like duality is developed for FIDL-modules, extending the dualities given in \cite{Celani1} for distributive lattices with fusion and implication and in \cite{Urquhart} for algebras of relevant logics. This duality enables us to describe the congruences of a FIDL-module and also to give a topological characterization of the simple and subdirectly irreducible FIDL-modules. 

The paper is organized as follows. In Section \ref{Preliminaries} we give some definitions and introduce the notations which are needed for the rest of the paper. In Section \ref{FIDL-modules} we introduce the class of modules with fusion and implication based over distributive lattices, or simply FIDL-modules. Also the concept of FIDL-subalgebra is developed and studied. In Section \ref{Representation of FIDL-modules} we study the notion of FIDL-homomorphism and we exhibit a representation theorem for FIDL-modules by means of relational structures. In Section \ref{Topological duality for FIDL-modules} we use the representation theorem and together with a suitable extension of the Priestley duality, we  obtain a duality for FIDL-modules as certain topological bi-spaces. Finally, in Section \ref{Congruences of FIDL-modules} we introduce the notion of congruence of FIDL-modules and as an application of the duality, we obtain a topological bi-spaced description for the simple and subdirectly irreducible FIDL-modules.

\section{Preliminaries} \label{Preliminaries}

Given a poset $\langle X, \leq \rangle$, a subset $U \subseteq X$ is said to be {\it{increasing}} ({\it{decreasing}}), if for every $x,y \in X$ such that $x \in U$ ($y \in U$) and $x \leq y$, then $y \in U$ ($x \in U$). The set of all increasing subsets of $X$ is denoted by $\mathcal{P}_{i}(X)$. For each $Y \subseteq X$, the increasing (decreasing) set generated by $Y$ is $[Y)=\{ x \in X \colon \exists y \in Y (y \leq x) \}$ ($(Y]=\{ x \in X \colon \exists y \in Y (x \leq y) \}$). If $Y=\{y\}$, then we will write $[y)$ and $(y]$ instead of $[\{y\})$ and $(\{y\}]$, respectively.

Given a bounded distributive lattice ${\bf{A}} = \langle A, \vee, \wedge, 0, 1 \rangle$, a set $F \subseteq A$ is called {\it{filter}} if $1 \in F$, $F$ is increasing, and if $a,b \in F$, then $a \wedge b \in F$. The {\it{filter generated by a subset $X \subseteq A$}} is the set 
\begin{equation*}
{\rm{Fig}}_{\bf{A}}(X)= \{ x \in A \colon \exists x_{1}, \hdots ,x_{n} \in X {\hspace{0.1cm}} {\text{such that}} {\hspace{0.1cm}} x_{1} \wedge \hdots \wedge x_{n} \leq x \}. 
\end{equation*}
If $X=\{a\}$, then ${\rm{Fig}}_{\bf{A}}(\{a\}) = [a)$. Denote by ${\rm{Fi}}(\textbf{A})$ the set of all filters of ${\bf{A}}$. A proper filter $P$ is {\it{prime}} if for every $a,b \in A$, $a \vee b \in P$ implies $a \in P$ or $b \in P$. We write $\mathcal{X}(\bf{A})$ the set of all prime filters of $\bf{A}$. Similarly, a set $I \subseteq A$ is called {\it{ideal}} if $0 \in I$, $I$ is decreasing, and if $a,b \in I$, then $a \vee b \in I$. Then the {\it{ideal generated by a subset $X \subseteq A$}} is the set 
\begin{equation*}
{\rm{Idg}}_{\bf{A}}(X)= \{ x \in A \colon \exists x_{1}, \hdots ,x_{n} \in X {\hspace{0.1cm}} {\text{such that}} {\hspace{0.1cm}} x \leq x_{1} \vee \hdots \vee x_{n} \}. 
\end{equation*}
In particular, if $X=\{a\}$, then ${\rm{Idg}}_{\bf{A}}(\{a\}) = (a]$. Denote by ${\rm{Id}}(\textbf{A})$ the set of all ideals of ${\bf{A}}$. Let $\beta_{\bf{A}} \colon A \to \mathcal{P}_{i}(\mathcal{X}(\bf{A}))$ be the map defined by $\beta_{\bf{A}}(a)=\{ P \in \mathcal{X}({\bf{A}}) \colon a \in P\}$. Then the family $\beta_{\bf{A}}[A]=\{ \beta_{\bf{A}}(a) \colon a \in A\}$ is closed under unions, intersections, and contains $\emptyset$ and $A$, i.e., it is a bounded distributive lattice. Moreover, $\beta_{\bf{A}}$ establishes an isomorphism between ${\bf{A}}$ and $\beta_{\bf{A}}[A]$.

A {\it{Priestley space}} is a triple $\langle X, \leq, \tau \rangle$ where $\langle X, \leq \rangle$ is a poset and $\langle X, \tau \rangle$ is a compact totally order-disconnected topological space. A morphism between Priestley spaces is a continuous and monotone function between them. If $\langle X, \leq, \tau \rangle$ is a Priestley space, then the family of all clopen increasing sets is denoted by $\mathcal{C}(X)$, and it is well known that $\mathcal{C}(X)$ is a bounded distributive lattice. The Priestley space of a bounded distributive lattice ${\bf{A}}$ is the triple $\langle \mathcal{X}({\bf{A}}), \subseteq_{\bf{A}}, \tau_{\bf{A}} \rangle$, where $\tau_{\bf{A}}$ is the topology generated by taking as a subbase the family $\{ \beta_{\bf{A}}(a) \colon a \in A \} \cup \{ \beta_{\bf{A}}(a)^{c} \colon a \in A \}$, where $\beta_{\bf{A}}(a)^{c} = \mathcal{X}({\bf{A}}) - \beta_{\bf{A}}(a)$. Therefore, ${\bf{A}}$ and $\mathcal{C}(\mathcal{X}({\bf{A}}))$ are isomorphic. If $\langle X, \leq, \tau \rangle$ is a Priestley space, then the map $\epsilon_{X} \colon X \to \mathcal{X}(\mathcal{C}(X))$ defined by $\epsilon_{X}(x) = \{ U \in \mathcal{C}(X) \colon x \in U \}$, for every $x \in X$, is a homeomorphism and an order-isomorphism. On the other hand, if $Y$ is a closed set of $\mathcal{X}({\bf{A}})$, then the relation 
\begin{equation} \label{congruence-closed}
\theta(Y)=\{(a,b)\in A\times A \colon \beta_{\textbf{A}}(a)\cap Y = \beta_{\textbf{A}}(b) \cap Y \}
\end{equation}
is a congruence of ${\bf{A}}$ and the correspondence $Y\rightarrow \theta(Y)$ establishes an anti-isomorphism between the lattice of closed subsets of $\mathcal{X}({\bf{A}})$ and the lattice of congruences of $\textbf{A}$. 

If $h \colon A \to B$ is a homomorphism between bounded distributive lattices ${\bf{A}}$ and ${\bf{B}}$, then the map $h^{*} \colon \mathcal{X}({\bf{B}}) \to \mathcal{X}({\bf{A}})$ defined by $h^{*}(P)=h^{-1}(P)$, for each $P \in \mathcal{X}({\bf{B}})$, is a continuous and monotone function. Conversely, if $\langle X, \leq_{X}, \tau_{X} \rangle$ and $\langle Y, \leq_{Y}, \tau_{Y} \rangle$ are Priestley spaces and $f \colon X \to Y$ is a continuous and monotone function, then the map $f^{*} \colon \mathcal{C}(Y) \to \mathcal{C}(X)$ defined by $f^{*}(U)=f^{-1}(U)$, for each $\mathcal{C}(Y)$, is a homomorphism between bounded distributive lattices. Furthermore, there is a duality between the algebraic category of bounded distributive lattices with homomorphisms and the category of Priestley spaces with continuous and monotone functions (\cite{Priestley,CLP,Cignoli}).

\section{FIDL-modules} \label{FIDL-modules}

In this section we present the class of \emph{modules with fusion and implication based over distributive lattices}, or \emph{FIDL-modules}, for short. These structures can be considered as bi-sorted distributive lattices endowed with two operations which preserve some of the lattice structure. We introduce the notion of FIDL-subalgebra and we exhibit a characterization of those in terms of some relations.

\begin{definition} \label{def_FIDL-modules}
Let ${\bf{A}}$, ${\bf{B}}$ be two bounded distributive lattices. A structure $\langle {\bf{A}}, {\bf{B}}, f \rangle$ is called a {\rm{FDL-module}}, if $f \colon A \times B \to A$ is a function such that for every $x,y \in A$ and every $b,c \in B$ the following conditions hold:
\begin{itemize}
\item[(F1)] $f(x \vee y, b)=f(x,b) \vee f(y,b)$,
\item[(F2)] $f(x, b \vee c)=f(x,b) \vee f(x,c)$,
\item[(F3)] $f(0,b)=0$,
\item[(F4)] $f(x,0)=0$.
\end{itemize}
A structure $\langle {\bf{A}}, {\bf{B}}, i \rangle$ is called an {\rm{IDL-module}}, if $i \colon B \times A \to A$ is a function such that for every $x,y \in A$ and every $b,c \in B$ the following conditions hold:
\begin{itemize}
\item[(I1)] $i(b, x \wedge y)=i(b,x) \wedge i(b,y)$,
\item[(I2)] $i(b \vee c, x)=i(b,x) \wedge i(c,x)$,
\item[(I3)] $i(b,1)=1$.
\end{itemize}
Moreover, a structure $\mathcal{M}=\langle {\bf{A}}, {\bf{B}}, f, i \rangle$ is called a {\rm{FIDL-module}}, if $\langle {\bf{A}}, {\bf{B}}, f \rangle$ is a FDL-module and $\langle {\bf{A}}, {\bf{B}}, i \rangle$ is an IDL-module.
\end{definition}

\begin{remark} \label{Fusion and impliction as unary operations}
Let $\mathcal{M}$ be a FIDL-module. Then the function $f$ determines and it is determined by a unique family $\mathcal{F}_{\bf{B}} = \{ f_{b} \colon A \to A \mid b \in B \}$ of unary operations on $\bf{A}$ such that for every $x,y \in A$ and every $b,c \in B$ the following conditions hold:
\begin{itemize}
\item[(F1')] $f_{b}(x \vee y)=f_{b}(x) \vee f_{b}(y)$,
\item[(F2')] $f_{b \vee c}(x)=f_{b}(x) \vee f_{c}(x)$,
\item[(F3')] $f_{b}(0)=0$,
\item[(F4')] $f_{0}(x)=0$.
\end{itemize}
Analogously, the function $i$ determines and it is determined by a unique family $\mathcal{I}_{\bf{B}} = \{ i_{b} \colon A \to A \mid b \in B \}$ of unary operations on $\bf{A}$ such that for every $x,y \in A$ and every $b,c \in B$ the following conditions hold:
\begin{itemize}
\item[(I1')] $i_{b}(x \wedge y)=i_{b}(x) \wedge i_{b}(y)$,
\item[(I2')] $i_{b \vee c}(x)=i_{b}(x) \wedge i_{c}(x)$,
\item[(I3')] $i_{b}(1)=1$.
\end{itemize}
Hence the FIDL-module $\mathcal{M}$ is equivalent to the structure $\langle {\bf{A}}, \mathcal{F}_{\bf{B}}, \mathcal{I}_{\bf{B}} \rangle$. Therefore, along this paper we will use the families $\mathcal{F}_{\bf{B}}$ and $\mathcal{I}_{\bf{B}}$ and its corresponding functions $f$ and $i$ indistinctly.
\end{remark}

The following are important examples of FIDL-modules.

\begin{example}
An algebra $\langle {\bf{A}}, \circ, \to \rangle$ is a {\it{bounded distributive lattice with fusion and implication}} (\cite{Celani1,Cabrer-Celani}), if ${\bf{A}}$ is a bounded distributive lattice and $\circ$ and $\to$ are binary operations defined on ${\bf{A}}$ such that for all $x,y,z \in A$ the following conditions hold: 
\begin{enumerate}
\item $x \circ (y \vee z) = (x \circ y) \vee (x \circ z)$,
\item $(x \vee y) \circ z = (x \circ z) \vee (y \circ z)$,
\item $x \circ 0 = 0 \circ x = 0$,
\item $x \to 1 = 1$, 
\item $(x \to y) \wedge (x \to z) = x \to (y \wedge z)$,
\item $(x \to z) \wedge (y \to z) = (x \vee y) \to z$. 
\end{enumerate}
Notice that if $\mathcal{M}$ is a FIDL-module such that $B=A$ and we consider the functions $x \circ_{f} y = f(x,y)$ and $x \to_{i} y = i(x,y)$, then $\langle {\bf{A}}, \circ , \to \rangle$ is a bounded distributive lattice with fusion and implication. Moreover, if $\mathcal{M}$ satisfies the condition $f(x,y) \leq z$ if and only if $x \leq i(y,z)$, then the structure $\langle {\bf{A}}, \circ , \to \rangle$ is a residuated lattice (\cite{JipsenTsinakis}).
\end{example}

\begin{example}
Recall that an algebra $\langle {\bf{A}}, \Box, \Diamond \rangle$ is a {\it{modal distributive lattice}}\footnote{Also called in \cite{Petrovich} \emph{distributive lattices with join and meet-homomorphisms}.}, or {\it{$\Box \Diamond$-lattice}}, if ${\bf{A}}$ is a bounded distributive lattice and $\Box$ and $\Diamond$ are unary operations defined on ${\bf{A}}$ such that for every $x,y \in A$ we have $\Box 1 = 1$, $\Box(x \wedge y)= \Box(x) \wedge \Box(y)$, $\Diamond 0 = 0$ and $\Diamond (x \vee y) = \Diamond (x) \vee \Diamond (y)$ (\cite{Chagrov,Venema,C2005}). If $\mathcal{M}$ is a FIDL-module and $B=\{0,1\}$, we can consider the functions $\Diamond_{f} (x) = f(x,1)$ and $\Box_{i} (x) = i(1,x)$ such that $\langle {\bf{A}}, \Diamond, \Box \rangle$ is a $\Box \Diamond$-lattice.
\end{example}

\begin{example} 
Let $\langle {\bf{A}}, \to \rangle$ be a Heyting algebra, where ${\bf{A}}$ is its bounded lattice reduct. Let $X$ be a non-empty set and let ${\bf{A}}^{X} = \langle A^{X}, \vee, \wedge, 0, 1 \rangle$ be the bounded distributive lattice of functions from $X$ to $A$ with the operations defined pointwise. Then, by following the notation of Remark \ref{Fusion and impliction as unary operations}, if we consider the families of functions $\mathcal{F}_{\bf{A}}=\{f_{a} \colon A^{X} \to A^{X} \mid a \in A\}$ and $\mathcal{I}_{\bf{A}}=\{i_{a} \colon A^{X} \to A^{X} \mid a \in A\}$ defined for every $a \in A$ by $f_{a}(g)(x) = a \wedge g(x)$ and $i_{a}(g)(x) = a \to g(x)$, respectively, it is the case that $\langle {\bf{A}}^{X}, \mathcal{F}_{\bf{A}}, \mathcal{I}_{\bf{A}} \rangle$ is a FIDL-module.
\end{example}

The following results are inspired by \cite{Celani1}.

\begin{prop} \label{propo_1}
Let $\mathcal{M}$ be a FIDL-module. Then for every $x,y \in A$ and every $b,c \in B$, if $x \leq y$ and $b \leq c$, then $f(x,b) \leq f(y,c)$ and $i(c,x) \leq i(b,y)$.
\end{prop}

\begin{proof}
Since $y=y \vee x$ and $c=b \vee c$, then by (F1) and (F2) of Definition \ref{def_FIDL-modules}
\begin{equation*}
f(y,c)=f(y \vee x, b \vee c)=f(y,b) \vee f(y,c) \vee f(x,b) \vee f(x,c) \geq f(x,b),
\end{equation*}
i.e., $f(x,b) \leq f(y,c)$. Analogously, as $x= x \wedge y$, by (I1) and (I2) of Definition \ref{def_FIDL-modules} we have
\begin{equation*}
i(c,x)=i(b \vee c, x \wedge y)=i(b,x) \wedge i(b,y) \wedge i(c,x) \wedge i(c,y) \leq i(b,y)
\end{equation*}
and $i(c,x) \leq i(b,y)$.
\end{proof}

Let $\mathcal{M}$ be a FIDL-module. Let $G \in {\rm{Fi}}(\bf{A})$ and $H \in {\rm{Fi}}(\bf{B})$. We define the following subsets:
\begin{equation*}
f(G,H)=\{x \in A \colon \exists (g,h) \in G \times H {\hspace{0.1cm}} {\text{such that}} {\hspace{0.1cm}} f(g,h) \leq x \}
\end{equation*}
and 
\begin{equation*}
i(H,G)=\{x \in A \colon \exists (h,g) \in H \times G {\hspace{0.1cm}} {\text{such that}} {\hspace{0.1cm}} g \leq i(h,x) \}.
\end{equation*}

\begin{prop}
Let $\mathcal{M}$ be a FIDL-module. If $G \in {\rm{Fi}}(\bf{A})$ and $H \in {\rm{Fi}}(\bf{B})$, then $f(G,H), i(H,G) \in {\rm{Fi}}(\bf{A})$.
\end{prop}

\begin{proof}
We prove that $f(G,H) \in {\rm{Fi}}(\bf{A})$. It is clear that $1 \in f(G,H)$ and $f(G,H)$ is increasing. If $x,y \in f(G,H)$, then there exist $(g,h), (\hat{g}, \hat{h}) \in G \times H$ such that $f(g,h) \leq x$ and $f(\hat{g}, \hat{h}) \leq y$. Since $G$ and $H$ are filters, $\bar{g}=g \wedge \hat{g} \in G$ and $\bar{h}=h \wedge \hat{h} \in H$. By Proposition \ref{propo_1}, $f(\bar{g}, \bar{h}) \leq x$ and $f(\bar{g}, \bar{h}) \leq y$. So, $f(\bar{g}, \bar{h}) \leq x \wedge y$ and $x \wedge y \in f(G,H)$. Then $f(G,H) \in {\rm{Fi}}(\bf{A})$. The proof for $i(H,G) \in {\rm{Fi}}(\bf{A})$ is similar.
\end{proof}

\begin{theorem} \label{theo_1}
Let $\mathcal{M}$ be a FIDL-module. Let $G \in {\rm{Fi}}(\bf{A})$, $H \in {\rm{Fi}}(\bf{B})$ and $P \in {\mathcal{X}}({\bf{A}})$. Then:
\begin{enumerate}
\item If $f(G,H) \subseteq P$, then there exist $Q \in {\mathcal{X}}({\bf{A}})$ and $R \in {\mathcal{X}}({\bf{B}})$ such that $G \subseteq Q$, $H \subseteq R$ and $f(Q,R) \subseteq P$.
\item If $i(H,G) \subseteq P$, then there exist $R \in {\mathcal{X}}({\bf{B}})$ and $Q \in {\mathcal{X}}({\bf{A}})$ such that $H \subseteq R$, $G \subseteq Q$ and $i(R,Q) \subseteq P$.
\end{enumerate}
\end{theorem}

\begin{proof}
We prove only $(1)$ because the proof of $(2)$ is analogous. Let us consider the family
\begin{equation*}
\mathcal{J} = \{ (K,W) \in {\rm{Fi}}({\bf{A}}) \times {\rm{Fi}}({\bf{B}}) \colon G \subseteq K, H \subseteq W \hspace{0.1cm} {\text{and}} \hspace{0.1cm} f(K,W) \subseteq P \}.
\end{equation*}
Since $(G,H) \in \mathcal{J}$, then $\mathcal{J} \neq \emptyset$. Observe that the union of a chain of elements of $\mathcal{J}$ is also in $\mathcal{J}$. So, by Zorn's Lemma, there is a maximal element $(Q,R) \in \mathcal{J}$. We see that $(Q,R) \in {\mathcal{X}}({\bf{A}}) \times {\mathcal{X}}({\bf{B}})$. Let $x,y \in A$ be such that $x \vee y \in Q$. Suppose that $x,y \notin Q$. Consider the filters $F_{x} = {\rm{Fig}}_{{\bf{A}}}(Q \cup \{x\})$ and $F_{y} = {\rm{Fig}}_{{\bf{A}}}(Q \cup \{y\})$. Then $Q \subset F_{x}$ and $Q \subset F_{y}$, and since $(Q,R)$ is maximal in $\mathcal{J}$, it follows that $f\left( F_{x}, R \right) \nsubseteq P$ and $f\left( F_{y}, R \right) \nsubseteq P$, i.e., there is $z \in f(F_{x}, R)$ such that $z \notin P$ and there is $t \in f(F_{y}, R)$ such that $t \notin P$. Then there exist $(f_{1}, r_{1}) \in F_{x} \times R$ and $(f_{2}, r_{2}) \in F_{y} \times R$ such that $f(f_{1}, r_{1}) \leq z$ and $f(f_{2}, r_{2}) \leq t$. So, there are $q_{1}, q_{2} \in Q$ such that $q_{1} \wedge x \leq f_{1}$ and $q_{2} \wedge y \leq f_{2}$. We take $q=q_{1} \wedge q_{2} \in Q$ and $r=r_{1} \wedge r_{2} \in R$. By Proposition \ref{propo_1}, we have $f(q \wedge x, r) \leq z$ and $f(q \wedge y, r) \leq t$. Thus, 
\begin{equation*}
f(q \wedge x, r) \vee f(q \wedge y, r) = f \left( (q \wedge x) \vee (q \wedge y), r \right) = f \left( q \wedge (x \vee y), r \right) \leq z \vee t.
\end{equation*}
As $q, x \vee y \in Q$, then $q \wedge (x \vee y) \in Q$ and $z \vee t \in f(Q,R)$. On the other hand, since $f(Q,R) \subseteq P$, we have $z \vee t \in P$. As $P$ is prime, $z \in P$ or $t \in P$ which is a contradiction. Then $Q \in {\mathcal{X}}({\bf{A}})$. The proof for $R \in {\mathcal{X}}({\bf{B}})$ is similar. It follows that there exist $Q \in {\mathcal{X}}({\bf{A}})$ and $R \in {\mathcal{X}}({\bf{B}})$ such that $G \subseteq Q$, $H \subseteq R$ and $f(Q,R) \subseteq P$.
\end{proof}

Let $\mathcal{M}$ be a FIDL-module. We define the following relations $R_{\mathcal{M}} \subseteq {\mathcal{X}}({\bf{A}}) \times {\mathcal{X}}({\bf{B}}) \times {\mathcal{X}}({\bf{A}})$ and $T_{\mathcal{M}} \subseteq {\mathcal{X}}({\bf{B}}) \times {\mathcal{X}}({\bf{A}}) \times {\mathcal{X}}({\bf{A}})$ by
\begin{equation}  \label{relation_R_{A}}
(Q,R,P) \in R_{\mathcal{M}} \Longleftrightarrow f(Q,R) \subseteq P,
\end{equation}
and 
\begin{equation} \label{relation_T_{A}}
(R,P,Q) \in T_{\mathcal{M}} \Longleftrightarrow i(R,P) \subseteq Q.
\end{equation}

\begin{lemma} \label{lem_1}
Let $\mathcal{M}$ be a FIDL-module. Let $x \in A$, $b \in B$ and $P \in {\mathcal{X}}({\bf{A}})$. Then: 
\begin{enumerate}
\item $f(x,b) \in P$ if and only if there exist $Q \in {\mathcal{X}}({\bf{A}})$ and $R \in {\mathcal{X}}({\bf{B}})$ such that $(Q,R,P) \in R_{\mathcal{M}}$, $x \in Q$ and $b \in R$.
\item $i(b,x) \in P$ if and only if for every $R \in {\mathcal{X}}({\bf{B}})$ and every $Q \in {\mathcal{X}}({\bf{A}})$, if $(R,P,Q) \in T_{\mathcal{M}}$ and $b \in R$, then $x \in Q$.
\end{enumerate}
\end{lemma}

\begin{proof}
$(1)$ Suppose $f(x,b) \in P$. We see that $f\left( [x), [b) \right) \subseteq P$. If $y \in f\left( [x), [b) \right)$, then there exists $(g,h) \in [x) \times [b)$ such that $f(g,h) \leq y$. So, $x \leq g$ and $b \leq h$, and by Proposition \ref{propo_1}, $f(x,b) \leq f(g,h) \leq y$. Since $P$ is a filter, $y \in P$ and $f\left( [x), [b) \right) \subseteq P$. So, by Theorem \ref{theo_1}, there exist $Q \in {\mathcal{X}}({\bf{A}})$ and $R \in {\mathcal{X}}({\bf{B}})$ such that $[x) \subseteq Q$, $[b) \subseteq R$ and $f(Q,R) \subseteq P$, i.e., $x \in Q$, $b \in R$ and $(Q,R,P) \in R_{\mathcal{M}}$. Conversely, if there exist $Q \in {\mathcal{X}}({\bf{A}})$ and $R \in {\mathcal{X}}({\bf{A}})$ such that $f(Q,R) \subseteq P$, $x \in Q$ and $b \in R$, because $(x,b) \in Q \times R$, we have $f(x,b) \in f(Q,R)$ and $f(x,b) \in P$.

$(2)$ Suppose $i(b,x) \in P$. Let $R \in {\mathcal{X}}(\textbf{B})$ and $Q \in {\mathcal{X}}(\textbf{A})$ be such that $i(R,P) \subseteq Q$ and $b \in R$. Then $(b, i(b,x)) \in R \times P$ and $x \in i(R,P)$. So, $x \in Q$. Reciprocally, suppose $i(b,x) \notin P$. We prove that $i\left( [b), P \right) \cap (x] = \emptyset$. Otherwise, there is $y \in i \left( [b), P \right)$ such that $y \in (x]$. Thus, there exists $(z,p) \in [b) \times P$ such that $p \leq i(z,y)$. Since $y \leq x$ and $b \leq z$, by Proposition \ref{propo_1}, we have $i(z,y) \leq i(b,x)$. Then $p \leq i(b,x)$ and $i(b,x) \in P$, which is a contradiction. So, $i\left( [b), P \right) \cap (x] = \emptyset$ and since $i\left( [b), P \right) \in {\rm{Fi}}(\textbf{A})$, by the Prime Filter Theorem there exists $Q \in {\mathcal{X}}(\textbf{A})$ such that $i \left( [b), P \right) \subseteq Q$ and $x \notin Q$. Then, by Theorem \ref{theo_1}, there exist $R \in {\mathcal{X}}(\textbf{B})$ and $\hat{P} \in {\mathcal{X}}(\textbf{A})$ such that $[b) \subseteq R$, $P \subseteq \hat{P}$ and $i(R, \hat{P}) \subseteq Q$. It is clear that $i(R,P) \subseteq i(R, \hat{P})$. Summarizing, there exist $R \in {\mathcal{X}}(\textbf{B})$ and $Q \in {\mathcal{X}}(\textbf{A})$ such that $(R,P,Q) \in T_{\mathcal{M}}$, $b \in R$ and $x \notin Q$, which contradicts the hypothesis. Therefore, $i(b,x) \in P$.
\end{proof}

Now, we introduce the concept of subalgebra of a FIDL-module.

\begin{definition}
Let $\mathcal{M}$ be a FIDL-module. Let $\hat{\bf{A}}$ be a bounded sublattice of $\bf{A}$ and $\hat{\bf{B}}$ a bounded sublattice of $\bf{B}$. 
\begin{itemize}
\item[(S1)] A structure $\langle \hat{\bf{A}}, \hat{\bf{B}}, f \rangle$ is called a {\rm{FDL-subalgebra of $\langle {\bf{A}}, {\bf{B}}, f \rangle$}}, if for every ${\hat{x}} \in {\hat{A}}$ and every ${\hat{b}} \in {\hat{B}}$, we have $f( \hat{x},\hat{b}) \in {\hat{A}}$. 
\item[(S2)] A structure $\langle \hat{\bf{A}}, \hat{\bf{B}}, i \rangle$ is called an {\rm{IDL-subalgebra of $\langle {\bf{A}}, {\bf{B}}, i \rangle$}}, if for every ${\hat{x}} \in {\hat{A}}$ and every ${\hat{b}} \in {\hat{B}}$, we have $i({\hat{b}},\hat{x}) \in {\hat{A}}$. 
\end{itemize}
Moreover, a structure $\hat{\mathcal{M}}=\langle \hat{\bf{A}}, \hat{\bf{B}}, f, i \rangle$ is called a {\rm{FIDL-subalgebra of $\mathcal{M}$}}, if $\langle \hat{\bf{A}}, \hat{\bf{B}}, f \rangle$ is a FDL-subalgebra and $\langle \hat{\bf{A}}, \hat{\bf{B}}, i \rangle$ is an IDL-subalgebra.
\end{definition}

We conclude this section with a characterization of FIDL-subalgebras by means of the relations defined in (\ref{relation_R_{A}}) and (\ref{relation_T_{A}}). 

\begin{theorem}
Let $\mathcal{M}$ be a FIDL-module. Let $\hat{\bf{A}}$ be a bounded sublattice of $\bf{A}$ and $\hat{\bf{B}}$ a bounded sublattice of $\bf{B}$. Then:
\begin{enumerate}
\item $\langle \hat{\bf{A}}, \hat{\bf{B}}, f \rangle$ is a FDL-subalgebra of $\langle {\bf{A}}, {\bf{B}}, f \rangle$ if and only if for all $P, Q, Q_{1} \in {\mathcal{X}}({\bf{A}})$ and for all $R_{1} \in {\mathcal{X}}({\bf{B}})$, if $(Q_{1}, R_{1}, P)\in R_{\mathcal{M}}$ and $P \cap \hat{A} \subseteq Q$, then there exist $Q_{2} \in {\mathcal{X}}({\bf{A}})$ and $R_{2} \in {\mathcal{X}}({\bf{B}})$ such that $Q_{1} \cap \hat{A} \subseteq Q_{2}$, $R_{1} \cap \hat{B} \subseteq R_{2}$ and $(Q_{2}, R_{2}, Q)\in R_{\mathcal{M}}$.

\item $\langle \hat{\bf{A}}, \hat{\bf{B}}, i \rangle$ is an IDL-subalgebra of $\langle {\bf{A}}, {\bf{B}}, i \rangle$ if and only if for all $P, Q, Q_{1} \in {\mathcal{X}}({\bf{A}})$ and for all $R_{1} \in {\mathcal{X}}({\bf{B}})$, if $(R_{1}, Q, Q_{1})\in T_{\mathcal{M}}$ and $P \cap \hat{A} \subseteq Q$, then there exist $Q_{2} \in {\mathcal{X}}({\bf{A}})$ and $R_{2} \in {\mathcal{X}}({\bf{B}})$ such that $Q_{2} \cap \hat{A} \subseteq Q_{1}$, $R_{1} \cap \hat{B} \subseteq R_{2}$ and $(R_{2}, P,Q_{2})\in T_{\mathcal{M}}$.
\end{enumerate}

We conclude that the structure $\hat{\mathcal{M}}=\langle \hat{\bf{A}}, \hat{\bf{B}}, f, i \rangle$ is a FIDL-subalgebra of $\mathcal{M}$ if and only if verifies the conditions $(1)$ and $(2)$.
\end{theorem}

\begin{proof}
$(1)$ Let $P, Q, Q_{1} \in {\mathcal{X}}({\bf{A}})$ and $R_{1} \in {\mathcal{X}}({\bf{B}})$ be such that $(Q_{1}, R_{1}, P)\in R_{\mathcal{M}}$ and $P \cap \hat{A} \subseteq Q$. Then, $f(Q_{1}, R_{1}) \subseteq P$. Consider the filters $F_{Q_{1}} = {\rm{Fig}}_{\bf{A}} (Q_{1} \cap \hat{A})$ and $F_{R_{1}} = {\rm{Fig}}_{\bf{B}} (R_{1} \cap \hat{B})$. It follows that $f \left( F_{Q_{1}}, F_{R_{1}}\right) \subseteq Q$. Indeed, if $x \in f \left( F_{Q_{1}}, F_{R_{1}} \right)$, then there exists $(g,h) \in F_{Q_{1}} \times F_{R_{1}}$ such that $f(g,h) \leq x$. So, there is $q_{1} \in Q_{1} \cap \hat{A}$ such that $q_{1} \leq g$ and there is $r_{1} \in R_{1} \cap \hat{B}$ such that $r_{1} \leq h$. By Proposition \ref{propo_1}, $f(q_{1}, r_{1}) \leq f(g,h) \leq x$. So, $f(q_{1}, r_{1}) \in f(Q_{1}, R_{1})$ and $f(q_{1}, r_{1}) \in P$. On the other hand, since $\langle \hat{\mathbf{A}}, \hat{\mathbf{B}}, f \rangle$ is a FDL-subalgebra, $f(q_{1}, r_{1}) \in \hat{A}$. Thus, $f(q_{1}, r_{1}) \in Q$ and $x \in Q$. Therefore, $f \left( F_{Q_{1}}, F_{R_{1}}\right) \subseteq Q$ and by Theorem \ref{theo_1}, there exist $Q_{2} \in {\mathcal{X}}({\bf{A}})$ and $R_{2} \in {\mathcal{X}}({\bf{B}})$ such that $Q_{1} \cap \hat{A} \subseteq Q_{2}$, $R_{1} \cap \hat{B} \subseteq R_{2}$ and $(Q_{2}, R_{2}, Q)\in R_{\mathcal{M}}$.

Conversely, suppose there exist $\hat{x} \in \hat{A}$ and $\hat{b} \in \hat{B}$ such that $f(\hat{x}, \hat{b}) \notin \hat{A}$. We prove that ${\rm{Fig}}_{\bf{A}} ( [f(\hat{x}, \hat{b})) \cap \hat{A} ) \cap (f(\hat{x}, \hat{b})] = \emptyset$. Otherwise, there is $y \in A$ such that $y \in {\rm{Fig}}_{\bf{A}} ( [f(\hat{x}, \hat{b})) \cap \hat{A})$ and $y \leq f(\hat{x}, \hat{b})$. Then there exists $z \in [f(\hat{x}, \hat{b})) \cap \hat{A}$ such that $z \leq y$. It follows that $f(\hat{x}, \hat{b})=z$. Since $z \in \hat{A}$, we have $f(\hat{x}, \hat{b}) \in \hat{A}$ which is a contradiction. Then ${\rm{Fig}}_{\bf{A}} ( [f(\hat{x}, \hat{b})) \cap \hat{A} ) \cap (f(\hat{x}, \hat{b})] = \emptyset$ and by the Prime Filter Theorem there exists $Q \in {\mathcal{X}}({\bf{A}})$ such that $[f(\hat{x}, \hat{b})) \cap \hat{A}\subseteq Q$ and $f(\hat{x}, \hat{b}) \notin Q$. It is easy to see that $[f(\hat{x}, \hat{b})) \cap {\rm{Idg}}_{\bf{A}} ( Q^{c} \cap \hat{A} ) = \emptyset$. Then there exists $P \in {\mathcal{X}}({\bf{A}})$ such that $f(\hat{x}, \hat{b}) \in P$ and $P \cap {\rm{Idg}}_{\bf{A}} ( Q^{c} \cap \hat{A} ) = \emptyset$, i.e., $P \cap \hat{A} \subseteq Q$. Since $f(\hat{x}, \hat{b}) \in P$, by Lemma \ref{lem_1} there exist $Q_{1} \in {\mathcal{X}}({\bf{A}})$ and $R_{1} \in {\mathcal{X}}({\bf{B}})$ such that $f(Q_{1},R_{1}) \subseteq P$, $\hat{x} \in Q_{1}$ and $\hat{b} \in R_{1}$. So $(Q_{1},R_{1},P)\in R_{\mathcal{M}}$. By assumption, there exist $Q_{2} \in {\mathcal{X}}({\bf{A}})$ and $R_{2} \in {\mathcal{X}}({\bf{B}})$ such that $Q_{1} \cap \hat{A} \subseteq Q_{2}$, $R_{1} \cap \hat{B} \subseteq R_{2}$ and $f(Q_{2}, R_{2}) \subseteq Q$. Thus, $(Q_{2},R_{2},Q)\in R_{\mathcal{M}}$. Since $\hat{x} \in \hat{A}$ and $\hat{b} \in \hat{B}$, we have $\hat{x} \in Q_{2}$ and $\hat{b} \in R_{2}$. Hence, $f(\hat{x}, \hat{b}) \in f (Q_{2}, R_{2})$ and $f(\hat{x}, \hat{b}) \in Q$, which is a contradiction. Therefore, $f(\hat{x}, \hat{b}) \in \hat{A}$ and we conclude that $\langle \hat{\bf{A}}, \hat{\bf{B}}, f \rangle$ is a FDL-subalgebra.

$(2)$ Let $P, Q, Q_{1} \in {\mathcal{X}}(\textbf{A})$ and $R_{1} \in {\mathcal{X}}(\textbf{A})$ be such that $(R_{1}, Q,Q_{1})\in T_{\mathcal{M}}$ and $P \cap \hat{A} \subseteq Q$. So $i(R_{1}, Q) \subseteq Q_{1}$. We see that 
\begin{equation*}
{\rm{Fig}}_{\textbf{A}} ( i ( {\rm{Fig}}_{\textbf{B}} (R_{1} \cap \hat{B}), P ) \cap \hat{A} ) \subseteq Q_{1}.
\end{equation*}
If $x \in {\rm{Fig}}_{\textbf{A}} ( i ( {\rm{Fig}}_{\textbf{B}} (R_{1} \cap \hat{B}), P ) \cap \hat{A} )$, then there is $y \in i ( {\rm{Fig}}_{\textbf{B}} (R_{1} \cap \hat{B}), P ) \cap \hat{A}$ such that $y \leq x$. So, there are $r \in R_{1} \cap \hat{B}$ and $p \in P$ such that $p \leq i(r,y)$. Thus, $i(r,y) \in P$. On the other hand, as $y \in \hat{A}$, $r \in \hat{B}$ and $\langle \hat{\mathbf{A}}, \hat{\mathbf{B}}, i \rangle$ is an IDL-subalgebra, $i(r,y) \in \hat{A}$. Then $i(r,y) \in P \cap \hat{A}$ and $i(r,y) \in Q$. It follows that $y \in i(R_{1},Q)$ and $y \in Q_{1}$. Then $x \in Q_{1}$. Now, let us consider the family
\begin{equation*}
\mathcal{J} = \{ F \in {\rm{Fi}}(\textbf{A}) \colon i ( {\rm{Fig}}_{\textbf{B}} (R_{1} \cap \hat{B}), P ) \subseteq F \hspace{0.1cm} {\text{and}} \hspace{0.1cm}  F \cap \hat{A} \subseteq Q_{1} \}.
\end{equation*}
Then $\mathcal{J} \neq \emptyset$ and by Zorn's Lemma there exists an maximal element $Q_{2} \in \mathcal{J}$. We prove that $Q_{2} \in {\mathcal{X}}(\textbf{A})$. Let $x,y \in A$ be such that $x \vee y \in Q_{2}$ and suppose $x,y \notin Q_{2}$. We take the filters $F_{x} = {\rm{Fig}}_{\textbf{A}} ( Q_{2} \cup \{x\} )$ and $F_{y} = {\rm{Fig}}_{\textbf{A}} ( Q_{2} \cup \{y\} )$. Then $F_{x} \cap \hat{A} \nsubseteq Q_{1}$ and $F_{y} \cap \hat{A} \nsubseteq Q_{1}$, i.e., there exist $z \in F_{x} \cap \hat{A}$ and $t \in F_{y} \cap \hat{A}$ such that $z,t \notin Q_{1}$. So, there are $q_{1}, q_{2} \in Q_{2}$ such that $q_{1} \wedge x \leq z$ and $q_{2} \wedge y \leq t$. Then $(q_{1} \wedge q_{2}) \wedge (x \vee y) \leq z \vee t$ and $z \vee t \in Q_{2}$. Since $\hat{A}$ is a sublattice, $z \vee t \in Q_{2} \cap \hat{A}$ and $z \vee t \in Q_{1}$, which is a contradiction because $Q_{1}$ is prime and $z \vee t \notin Q_{1}$. Hence, $Q_{2} \in {\mathcal{X}}(\textbf{A})$. As $i ( {\rm{Fig}}_{\textbf{B}} (R_{1} \cap \hat{B}), P ) \subseteq Q_{2}$ and $F \cap \hat{A} \subseteq Q_{1}$, by Theorem \ref{theo_1} there exists $R_{2} \in {\mathcal{X}}(\textbf{B})$ such that $R_{1} \cap \hat{B} \subseteq R_{2}$ and $(R_{2},P,Q_{2})\in T_{\mathcal{M}}$.

Reciprocally, suppose there exist $\hat{x} \in \hat{A}$ and $\hat{b} \in \hat{B}$ such that $i(\hat{b}, \hat{x}) \notin \hat{A}$. In order to prove our claim, first we show that ${\rm{Idg}}_{\textbf{A}} ( (i(\hat{b}, \hat{x})] \cap \hat{A} ) \cap [i(\hat{b}, \hat{x})) = \emptyset$. If there is $y \in {\rm{Idg}}_{\textbf{A}} ( (i(\hat{b}, \hat{x})] \cap \hat{A} )$ such that $i(\hat{b}, \hat{x}) \leq y$, then there exists $z \in (i(\hat{b}, \hat{x})] \cap \hat{A}$ such that $y \leq z$. Thus, $i(\hat{b}, \hat{x})=z$ and $i(\hat{b}, \hat{x}) \in \hat{A}$, which is a contradiction. Then ${\rm{Idg}}_{\textbf{A}} ( (i(\hat{b}, \hat{x})] \cap \hat{A} ) \cap [i(\hat{b}, \hat{x})) = \emptyset$ and consequently from the Prime Filter Theorem, there exists $P \in {\mathcal{X}}(\textbf{A})$ such that $i(\hat{b}, \hat{x}) \in P$ and ${\rm{Idg}}_{\textbf{A}} ( (i(\hat{b}, \hat{x})] \cap \hat{A} ) \cap P = \emptyset$. It is easy to prove that $(i(\hat{b}, \hat{x})] \cap {\rm{Fig}}_{\textbf{A}}(P \cap \hat{A}) = \emptyset$. Then again by the Prime Filter Theorem, there is $Q \in {\mathcal{X}}(\textbf{A})$ such that $P \cap \hat{A} \subseteq Q$ and $i(\hat{b}, \hat{x}) \notin Q$. By Lemma \ref{lem_1}, there exist $R_{1} \in {\mathcal{X}}(\textbf{B})$ and $Q_{1} \in {\mathcal{X}}(\textbf{A})$ such that $i(R_{1},Q) \subseteq Q_{1}$, $\hat{b} \in R_{1}$ and $\hat{x} \notin Q_{1}$. So, by hypothesis, there exist $Q_{2} \in {\mathcal{X}}(\textbf{A})$ and $R_{2} \in {\mathcal{X}}(\textbf{B})$ such that $Q_{2} \cap \hat{A} \subseteq Q_{1}$, $R_{1} \cap \hat{B} \subseteq R_{2}$ and $i(R_{2}, P) \subseteq Q_{2}$. Thus, $(R_{2}, P, Q_{2})\in T_{\mathcal{M}}$. As $\hat{b} \in R_{1}$, we have $\hat{b} \in R_{2}$. On the other hand, as $i(\hat{b}, \hat{x}) \in P$, we have $\hat{x} \in i(R_{2}, P)$ and $\hat{x} \in Q_{2}$. Then $\hat{x} \in Q_{2} \cap \hat{A}$ and $\hat{x} \in Q_{1}$ which is a contradiction. Hence, $i(\hat{b}, \hat{x}) \in \hat{A}$ and $\langle \hat{\bf{A}}, \hat{\bf{B}}, i \rangle$ is an IDL-subalgebra.
\end{proof}

\section{Representation for FIDL-modules} \label{Representation of FIDL-modules}

The main purpose of this section is to show a representation theorem for FIDL-modules in terms of certain relational structures consisting of bi-posets endowed with two relations.

We start by defining a category whose objects are FIDL-modules. So, we need to describe first, the notion of homomorphism between FIDL-modules. Recall that for every pair of functions $\alpha \colon A \to \hat{A}$ and $\gamma \colon B \to \hat{B}$ we can consider the map $\alpha \times \gamma \colon A \times B \to \hat{A} \times \hat{B}$ which is defined by $\left( \alpha \times \gamma \right) (x,y) = \left( \alpha(x), \gamma(y) \right)$.

\begin{definition} \label{Definition FIDL-homomorphism}
Let $\mathcal{M}=\langle {\bf{A}}, {\bf{B}}, f, i \rangle$ and $\hat{\mathcal{M}}=\langle \hat{\bf{A}}, \hat{\bf{B}}, \hat{f}, \hat{i} \rangle$ be two FIDL-modules. We shall say that a pair $\left( \alpha, \gamma \right) \colon \mathcal{M} \to \hat{\mathcal{M}}$ is a {\rm{FIDL-homomorphism}}, if $\alpha \colon A \to \hat{A}$ and $\gamma \colon B \to \hat{B}$ are homomorphisms between bounded distributive lattices and the following diagrams commute:
\begin{displaymath}
\begin{tabular}{ccc}
\xymatrix{
A \times B \ar[r]^-{f} \ar[d]_-{\alpha \times \gamma} & A \ar[d]^-{\alpha} \\
\hat{A} \times \hat{B} \ar[r]_-{\hat{f}} & \hat{A}
} & & \xymatrix{
B \times A \ar[r]^-{i} \ar[d]_-{\gamma \times \alpha} & A \ar[d]^-{\alpha} \\
\hat{B} \times \hat{A} \ar[r]_-{\hat{i}} & \hat{A}
}
\end{tabular}
\end{displaymath}
\end{definition}

\begin{remark} \label{Remark_f and i}
Notice that from Remark \ref{Fusion and impliction as unary operations}, the diagrams of Definition \ref{Definition FIDL-homomorphism} are commutative if and only if for every $b \in B$, the following diagrams commute:
\begin{displaymath}
\begin{tabular}{ccc}
\xymatrix{
A  \ar[r]^-{f_{b}} \ar[d]_-{\alpha} & A \ar[d]^-{\alpha} \\
\hat{A}  \ar[r]_-{{\hat{f}}_{\gamma(b)}} & \hat{A}
} & & \xymatrix{
A  \ar[r]^-{i_{b}} \ar[d]_-{\alpha} & A \ar[d]^-{\alpha} \\
\hat{A}  \ar[r]_-{{\hat{i}}_{\gamma(b)}} & \hat{A}
}
\end{tabular}
\end{displaymath}
We stress that for the rest of the paper we will use the functions ${\hat{f}}_{\gamma(b)}$ and ${\hat{i}}_{\gamma(b)}$ as well as the notation of Definition \ref{Definition FIDL-homomorphism} indistinctly.
\end{remark}

\begin{example}
Let $\mathcal{M}$ be a FIDL-module. Let ${\bf{C}}$ be a bounded distributive lattice and $h \colon C \to B$ be a lattice homomorphism. If we define the functions $\hat{f} \colon A \times C \to A$ by $\hat{f}(x,c)=f(x,h(c))$ and $\hat{i} \colon C \times A \to A$ by $\hat{i}(c,x)=i(h(c),x)$, then the structure $\mathcal{N} = \langle {\bf{A}}, {\bf{C}}, \hat{f}, \hat{i} \rangle$ is a FIDL-module and the pair $\left( id_{A}, h \right) \colon \mathcal{N} \to \mathcal{M}$ is a FIDL-homomorphism.
\end{example}

Let $\mathcal{M}=\langle {\bf{A}}, {\bf{B}}, f, i \rangle$, ${\hat{\mathcal{M}}}=\langle \hat{\bf{A}}, \hat{\bf{B}}, \hat{f}, \hat{i} \rangle$ and ${\bar{\mathcal{M}}} = \langle \bar{\bf{A}}, \bar{\bf{B}}, \bar{f}, \bar{i} \rangle$ be FIDL-modules. Consider the FIDL-homomorphisms $\left( \alpha, \gamma \right) \colon \mathcal{M} \to {\hat{\mathcal{M}}}$ and $\left( \delta, \lambda \right) \colon {\hat{\mathcal{M}}} \to {\bar{\mathcal{M}}}$. Then we define the composition $\left( \delta, \lambda \right)  \left( \alpha, \gamma \right) \colon \mathcal{M} \to {\bar{\mathcal{M}}}$ as the pair $\left( \delta \alpha, \lambda \gamma \right)$. It is clear that the FIDL-homomorphisms between FIDL-modules are closed by composition and that such a composition is associative. Moreover, we may define the identity of $\mathcal{M}$ as the pair $\left( id_{A}, id_{B} \right)$. So, we obtain that the class $\mathsf{FIMod}$ of FIDL-modules as objects and FIDL-homomorphisms as morphisms is a category. 

The following technical result will be useful later.

\begin{lemma} \label{Isos FDL}
Let $\mathcal{M}=\langle {\bf{A}}, {\bf{B}}, f, i \rangle$ and $\hat{\mathcal{M}}=\langle \hat{\bf{A}}, \hat{\bf{B}}, \hat{f}, \hat{i} \rangle$ be two FIDL-modules and $\left( \alpha, \gamma \right) \colon \mathcal{M} \to \hat{\mathcal{M}}$ a FIDL-homomorphism. Then the following conditions are equivalent: 
\begin{enumerate}
\item $\left( \alpha, \gamma \right)$ is an isomorphism in the category ${\mathsf{FIMod}}$, 
\item $\alpha$ and $\gamma$ are isomorphisms of bounded distributive lattices.
\end{enumerate}
\end{lemma}

\begin{proof}
$(1) \Rightarrow (2)$ Immediate.

$(2) \Rightarrow (1)$ Let us assume that $\alpha$ and $\gamma$ are isomorphism of bounded distributive lattices. We will show that the pair $\left( {\alpha}^{-1}, {\gamma}^{-1} \right) \colon \hat{\mathcal{M}} \to \mathcal{M}$ is a FIDL-homomorphism. Since ${\alpha}^{-1}$ and ${\gamma}^{-1}$ are isomorphisms of bounded distributive lattices, only remains to check the commutativity of the following diagram:
\begin{displaymath}
\xymatrix{
\hat{A} \times \hat{B} \ar[r]^-{{\hat{f}}} \ar[d]_-{{\alpha}^{-1} \times {\gamma}^{-1}} & \hat{A} \ar[d]^-{{\alpha}^{-1}} \\
A \times B \ar[r]_-{f} & A
}
\end{displaymath}
By hypothesis, we have ${\hat{f}} \left( \alpha \times \gamma \right) = \alpha f$. So, ${\alpha}^{-1} {\hat{f}} \left( \alpha \times \gamma \right) = f$ and 
\begin{equation*}
f \left( {\alpha}^{-1} \times {\gamma}^{-1} \right) = {\alpha}^{-1} {\hat{f}} \left( {\alpha} \times {\gamma} \right) \left( {\alpha}^{-1} \times {\gamma}^{-1} \right) = {\alpha}^{-1} {\hat{f}} \left( id_{\hat{A} \times \hat{B}} \right) = {\alpha}^{-1} {\hat{f}},
\end{equation*}
i.e., $f \left( {\alpha}^{-1} \times {\gamma}^{-1} \right) = {\alpha}^{-1} {\hat{f}}$. Similarly, the commutativity of the following diagram is easily verified:
\begin{displaymath}
\xymatrix{
\hat{B} \times \hat{A} \ar[r]^-{{\hat{i}}} \ar[d]_-{{\gamma}^{-1} \times {\alpha}^{-1}} & \hat{A} \ar[d]^-{{\alpha}^{-1}} \\
B \times A \ar[r]_-{i} & A
}
\end{displaymath}
By definition of the composition in ${\mathsf{FIMod}}$, it follows that $\left( \alpha, \gamma \right)^{-1} = \left( {\alpha}^{-1}, {\gamma}^{-1} \right)$ and $\left( \alpha, \gamma \right)$ is an isomorphism in the category ${\mathsf{FIMod}}$.
\end{proof}

Now, we introduce the class of relational structures needed to develop our representation theorem as well as the notion of morphisms between them.

\begin{definition}   \label{definition FI-frames}
Let $\langle X, \leq_{X} \rangle$ and $\langle Y, \leq_{Y} \rangle$ be two posets. A structure $\langle X, Y, \leq_{X}, \leq_{Y}, R \rangle$ is called a {\rm{F-frame}}, if $R \subseteq X \times Y \times X$ is a relation such that:
\begin{equation}   \label{condition_R}
\text{if} \hspace{0.1cm} (x,y,z) \in R, \bar{x} \leq_{X} x, \bar{y} \leq_{Y} y \hspace{0.1cm} \text{and} \hspace{0.1cm} z \leq_{X} \bar{z}, \hspace{0.1cm} \text{then} \hspace{0.1cm} (\bar{x},\bar{y},\bar{z}) \in R. 
\end{equation}
A structure $\langle X, Y, \leq_{X}, \leq_{Y}, T \rangle$ is called an {\rm{I-frame}}, if $T \subseteq Y \times X \times X$ is a relation such that:
\begin{equation}   \label{condition_T}
\text{if} \hspace{0.1cm} (y,x,z) \in T, \bar{y} \leq_{Y} y, \bar{x} \leq_{X} x \hspace{0.1cm} \text{and} \hspace{0.1cm} z \leq_{X} \bar{z}, \hspace{0.1cm} \text{then} \hspace{0.1cm} (\bar{y},\bar{x},\bar{z}) \in T. 
\end{equation}
Moreover, a structure $\mathcal{F}=\langle X, Y, \leq_{X}, \leq_{Y}, R, T \rangle$ is called a {\rm{FI-frame}}, if $\langle X, Y, \leq_{X}, \leq_{Y}, R \rangle$ is a F-frame and $\langle X, Y, \leq_{X}, \leq_{Y}, T \rangle$ is an I-frame.
\end{definition}

\begin{definition} \label{Definition FI-morphism}
Let $\mathcal{F}$ and $\hat{\mathcal{F}}$ be two FI-frames. We shall say that a pair $\left( g, h \right) \colon \mathcal{F} \to \hat{\mathcal{F}}$ is a {\rm{FI-morphism}}, if $g \colon X \to \hat{X}$ and $h \colon Y \to \hat{Y}$ are morphisms between posets and the following conditions hold:
\begin{itemize}
\item[(M1)] If $(x,y,z) \in R$, then $(g(x), h(y), g(z)) \in \hat{R}$.
\item[(M2)] If $(\bar{x}, \bar{y}, g(z)) \in \hat{R}$, then there exist $x \in X$ and $y \in Y$ such that $(x,y,z) \in R$, $\bar{x} \leq_{\hat{X}} g(x)$ and $\bar{y} \leq_{\hat{Y}} h(y)$.
\item[(N1)] If $(x,y,z) \in T$, then $(h(x), g(y), g(z)) \in \hat{T}$.
\item[(N2)] If $(\bar{x}, g(y), \bar{z}) \in \hat{T}$, then there exist $x \in Y$ and $z \in X$ such that $(x,y,z) \in T$, $\bar{x} \leq_{\hat{Y}} h(x)$ and $g(z) \leq_{\hat{X}} \bar{z}$.
\end{itemize}
\end{definition}

The composition of FI-morphisms is defined component-wise. It is clear from Definition \ref{Definition FI-morphism}, that such a composition is closed and associative and for every FI-Frame $\mathcal{F}$, the identity arrow is given by the pair $(id_{X},id_{Y})$. We write $\mathsf{FIFram}$ for the category of FI-frames and FI-morphisms.

The following result is similar to Lemma \ref{Isos FDL} and will be useful at the moment of proving the main theorem of this section.

\begin{lemma} \label{Isos FIF}
Let $\mathcal{F}$ and $\hat{\mathcal{F}}$ be two FI-frames and $\left( g, h \right) \colon \mathcal{F} \to \hat{\mathcal{F}}$ a FI-morphism. Then the following conditions are equivalent: 
\begin{enumerate}
\item $\left( g, h \right)$ is an isomorphism in the category ${\mathsf{FIFram}}$, 
\item $g$ and $h$ are isomorphisms of posets.
\end{enumerate}
\end{lemma}

\begin{proof}
$(1) \Rightarrow (2)$ Immediate.

$(2) \Rightarrow (1)$ Since $g$ and $h$ are isomorphisms of posets, then there exist $g^{-1}$ and $h^{-1}$. It is clear that $(g,h)^{-1} = (g^{-1}, h^{-1})$. We need to check that $(g^{-1}, h^{-1}) \colon \hat{\mathcal{F}} \to \mathcal{F}$ is a FI-morphism. We prove (M1). Let $(\bar{x}, \bar{y}, \bar{z}) \in \hat{X} \times \hat{Y} \times \hat{X}$ such that $(\bar{x}, \bar{y}, \bar{z}) \in \hat{R}$ and suppose that $(g^{-1}(\bar{x}), h^{-1}(\bar{y}), g^{-1}(\bar{z})) \notin R$. Due to $(\bar{x}, \bar{y}, g(g^{-1}(\bar{z}))) \in \hat{R}$ and $(g,h)$ is a FI-morphism, there exist $x \in X$ and $y \in Y$ such that $(x,y,g^{-1}(\bar{z})) \in R$, $\bar{x} \leq_{\hat{X}} g(x)$ and $\bar{y} \leq_{\hat{Y}} h(y)$. On the other hand, since $g^{-1}$ and $h^{-1}$ are monotone, we have $g^{-1}(\bar{x}) \leq_{X} x$, $h^{-1}(\bar{y}) \leq_{Y} y$ and $g^{-1}(\bar{z}) \leq_{X} g^{-1}(\bar{z})$. It follows, by (\ref{condition_R}), that $(g^{-1}(\bar{x}), h^{-1}(\bar{y}), g^{-1}(\bar{z})) \in R$ which is a contradiction. We prove (M2). Let $(x,y,g^{-1}(\bar{z})) \in R$. As $x=g^{-1}(g(x))$ and $y=h^{-1}(h(y))$, then $(g^{-1}(g(x)), h^{-1}(h(y)), g^{-1}(\bar{z})) \in R$. Since $(g,h)$ is a FI-morphism, $(g(x), h(y), \bar{z}) \in \hat{R}$. Conditions (N1) and (N2) can be verified analogously. 
\end{proof}

It is the moment to show how we build our representation. Let $\mathcal{F}=\langle X, Y, \leq_{X}, \leq_{Y}, R, T \rangle$ be a FI-frame. Then it follows that $\langle \mathcal{P}_{i}(X), \cup, \cap, \emptyset, X \rangle$ and $\langle \mathcal{P}_{i}(Y), \cup, \cap, \emptyset, Y \rangle$ are bounded distributive lattices. Let $U \in \mathcal{P}_{i}(X)$ and $V \in \mathcal{P}_{i}(Y)$, and let us to consider the following subsets of $X$:  
\begin{equation}    \label{f_in P}
f_{\mathcal{F}}(U,V)=\{ z \in X \colon \exists (x,y) \in U \times V {\hspace{0.1cm}} {\text{such that}} {\hspace{0.1cm}} (x,y,z) \in R \}
\end{equation}
and 
\begin{equation}  \label{i_in P}
i_{\mathcal{F}}(V,U)=\{ y \in X \colon \forall x \in Y, \forall z \in X {\hspace{0.05cm}} (((x,y,z) \in T {\hspace{0.1cm}} {\text{and}} {\hspace{0.1cm}} x \in V) {\hspace{0.1cm}} {\text{implies}} {\hspace{0.1cm}} z \in U) \}.
\end{equation}

It is easy to prove that $f_{\mathcal{F}}(U,V), i_{\mathcal{F}}(V,U) \in \mathcal{P}_{i}(X)$.

The proof of the following two results are routine so the details are left to the reader.

\begin{lemma} \label{representation}
Let $\mathcal{F}$ be a FI-frame. Then the structure 
\begin{equation*}
\mathcal{M}_{\mathcal{F}}=\langle \mathcal{P}_{i}(X), \mathcal{P}_{i}(Y), f_{\mathcal{F}}, i_{\mathcal{F}} \rangle
\end{equation*}
is a FIDL-module, where $f_{\mathcal{F}}$ and $i_{\mathcal{F}}$ are given by \ref{f_in P} and \ref{i_in P}, respectively.
\end{lemma}

\begin{lemma} \label{Representation of a FI-frame}
Let $ \mathcal{M}$ be a FIDL-module. Then the structure 
\begin{equation*}
\mathcal{F}_{\mathcal{M}}=\langle {\mathcal{X}}({\bf{A}}), {\mathcal{X}}({\bf{B}}), \subseteq_{\bf{A}}, \subseteq_{\bf{B}}, R_{\mathcal{M}}, T_{\mathcal{M}} \rangle
\end{equation*}
is a FI-frame, where $R_{\mathcal{M}}$ and $T_{\mathcal{M}}$ are given by \ref{relation_R_{A}} and \ref{relation_T_{A}}, respectively.
\end{lemma}

\begin{lemma} \label{Representation of FIDL morphisms} 
Let $\mathcal{M}$ and $\hat{\mathcal{M}}$ be two FIDL-modules. If $( \alpha, \gamma) \colon \mathcal{M}\to \hat{\mathcal{M}}$ is a FIDL-homomorphism, then $(\alpha^{*}, \gamma^{*}) \colon \mathcal{F}_{\hat{\mathcal{M}}} \to \mathcal{F}_{\mathcal{M}}$ is a FI-morphism.
\end{lemma}

\begin{proof}
We start by proving (M1). Let $(\hat{Q},\hat{R},\hat{P})\in R_{\hat{\mathcal{M}}}$. In order to prove $(\alpha^{\ast}(\hat{Q}),\gamma^{\ast}(\hat{R}),\alpha^{\ast}(\hat{P}))\in R_{\mathcal{M}}$, let $y\in f(\alpha^{\ast}(\hat{Q}),\gamma^{\ast}(\hat{R}))$. Then there exist $(a,b)\in \alpha^{\ast}(\hat{Q})\times \gamma^{\ast}(\hat{R})$ such that $f(a,b)\leq_{A} y$. Since $\alpha$ is monotone and $(\alpha,\gamma)$ is a FIDL-homomorphism, we have $\alpha(f(a,b))=\hat{f}(\alpha(a),\gamma(b))\leq_{\hat{A}} \alpha(a)$. Since $\alpha(a)\in P$ and $\gamma(b)\in R$, then $\alpha(y)\in \hat{f}(\hat{Q},\hat{R})$. So, our assumption allows us to conclude that $\alpha(y)\in \hat{P}$. Hence $f(\alpha^{\ast}(\hat{Q}),\gamma^{\ast}(\hat{R}))\subseteq_{\mathcal{X}(\bf{A})} \alpha^{\ast}(\hat{P})$. The proof of (N1) is similar. Now we prove (M2). Let us assume that $(Q,R,\alpha^{\ast}(\hat{P}))\in R_{\mathcal{M}}$. Note that, such an assumption allows us to say that $f(a,b)\in \alpha^{\ast}(\hat{P})$ for every $(a,b)\in Q\times R$. Then, since $(\alpha,\gamma)$ is a FIDL-homomorphism, it follows that $\hat{f}(\alpha(a),\gamma(b))\in P$. From $(1)$ of Lemma \ref{lem_1}, there exist $\hat{Q}\in \mathcal{X}(\hat{\mathbf{A}})$ and $\hat{R}\in \mathcal{X}(\hat{\mathbf{B}})$ such that $\alpha(a)\in Q$, $\gamma(b)\in R$ such that $(\hat{Q},\hat{R},\hat{P})\in R_{\hat{\mathcal{M}}}$. Hence (M2) holds. Finally, for proving (N2), let us assume that $(R,\alpha^{\ast}(\hat{P}),Q)\in T_{\mathcal{M}}$. Consider $F={\rm{Fig}}_{\hat{\bf{B}}}(\gamma(R))$ and $I={\rm{Idg}}_{\hat{\bf{A}}}(\alpha(Q^{c}))$. We see that $\hat{i}(F, \hat{P})\cap I=\emptyset$. Assume the contrary, then there exist $y\in \hat{A}$, $a\in \hat{P}$, $b\in \hat{B}$, $r\in R$ and $q\notin Q$ such that $a\leq_{\hat{A}} \hat{i}(b,y)$, $\gamma(r)\leq_{\hat{B}} b$ and $y\leq_{\hat{A}} \alpha(q)$. Since $(\alpha,\gamma)$ is a FIDL-homomorphism, by Proposition \ref{propo_1}, we obtain that $a\leq_{\hat{A}} \hat{i}(\gamma(r), \alpha(q))=\alpha(i(r,q))$. Hence $i(r,q)\in \alpha^{\ast}(\hat{P})$ so $q\in i(R,\alpha^{\ast}(\hat{P}))$ and therefore $q\in Q$, which is a contradiction. Then $\hat{i}(F, \hat{P})\cap I=\emptyset$ and by the Prime Filter Theorem, there exist $\hat{Q}\in \mathcal{X}(\hat{\mathbf{A}})$ such that $\hat{i}(F,\hat{P})\subseteq_{\mathcal{X}(\hat{\mathbf{A}})} \hat{Q}$ and $\hat{Q}\cap I=\emptyset$. On the other hand, from Theorem \ref{theo_1}, there exists $\hat{R}\in \mathcal{X}(\hat{\mathbf{B}})$ such that $\hat{i}(\hat{R},\hat{P})\subseteq_{\mathcal{X}(\hat{\mathbf{A}})} \hat{Q}$ and $R\subseteq_{\mathcal{X}(\mathbf{B})} \gamma^{\ast}(\hat{R})$. As $\alpha(Q^{c})\subseteq_{\mathcal{X}(\hat{\mathbf{A}})} I$ we get that $\hat{Q}\cap \alpha(Q^{c})=\emptyset$. It is not hard to see that the latter is equivalent to say that $\alpha^{\ast}(\hat{Q})\subseteq_{\mathcal{X}(\mathbf{A})} Q$.
\end{proof}

\begin{lemma}\label{Representation of FI-frame morphisms} 
Let $\mathcal{F}$ and $\hat{\mathcal{F}}$ be two FI-frames. If $(g, h) \colon \mathcal{F} \to \hat{\mathcal{F}}$ is a FI-morphism, then $(g^{*}, h^{*}) \colon \mathcal{M}_{\hat{\mathcal{F}}} \\ \to \mathcal{M}_{\mathcal{F}}$ is a FIDL-homomorphism.
\end{lemma}

\begin{proof}
Let $\mathcal{M}_{\mathcal{F}}$ and $\mathcal{M}_{\mathcal{\hat{F}}}$ be the FIDL-modules that arise from Lemma \ref{representation}. In order to simplify notation, in this proof we write $f$ and $i$ instead of $f_{\mathcal{F}}$. Similarly, we write $\hat{f}$ and $\hat{i}$ instead of $f_{\mathcal{\hat{F}}}$ and $i_{\mathcal{\hat{F}}}$. This is with the aim of setting our proof within the context of Remark \ref{Remark_f and i}. It is clear that $g^{\ast}:\mathcal{P}_{i}(\hat{X})\rightarrow \mathcal{P}_{i}(X)$ and $h^{\ast}:\mathcal{P}_{i}(\hat{Y})\rightarrow \mathcal{P}_{i}(Y)$ are homomorphisms of bounded distributive lattices so, in order to prove our claim we proceed to check that for every $U\in \mathcal{P}_{i}(\hat{Y})$, the following diagrams 
\begin{displaymath}
\begin{array}{ccc}
\xymatrix{
\mathcal{P}_{i}(\hat{X}) \ar[d]_-{g^{\ast}} \ar[r]^-{\hat{f}_{U}} & \mathcal{P}_{i}(\hat{X}) \ar[d]^-{g^{\ast}} \\
\mathcal{P}_{i}(X) \ar[r]_-{f_{h^{\ast}(U)}} & \mathcal{P}_{i}(X)
}
 & & 
\xymatrix{
\mathcal{P}_{i}(\hat{X}) \ar[d]_-{g^{\ast}} \ar[r]^-{\hat{i}_{U}} & \mathcal{P}_{i}(\hat{X}) \ar[d]^-{g^{\ast}} \\
\mathcal{P}_{i}(X) \ar[r]_-{i_{h^{\ast}(U)}} & \mathcal{P}_{i}(X)
}
\end{array}
\end{displaymath}
commute. For the diagram of the left, let $V\in \mathcal{P}_{i}(\hat{X})$ and $z\in g^{\ast}(\hat{f}_{U}(V))$. So, there exist $x'\in U$ and $y'\in V$ such that $(x',y',g(z))\in \hat{R}$. Since $(g,h)$ is a FI-morphism then from (M2) of Definition \ref{Definition FI-morphism}, there exist $x\in Y$ and $y\in X$ such that $x'\leq_{\hat{Y}} h(x)$, $y'\leq_{\hat{X}} g(y)$ and $(x,y,z)\in R$. Hence $x\in h^{\ast}(U)$, $y\in g^{\ast}(V)$ and $(x,y,z)\in R$. That is to say, $z\in f_{h^{\ast}(U)}(g^{\ast}(V))$. The other inclusion is straightforward. For the diagram of the right, let $V\in \mathcal{P}_{i}(\hat{X})$ and suppose that $y\in i_{h^{\ast}(U)}(g^{\ast}(V))$. So, for every $x\in Y$ and $z\in X$, if $(x,y,z)\in T$ and $h(x)\in U$, then $g(z)\in V$. We recall that for showing $y\in g^{\ast}(\hat{i}_{U}(V))$ we need to prove that for every  $x'\in \hat{Y}$ and $z'\in \hat{X}$ such that $(x',g(y),z')\in \hat{T}$ and $x'\in U$, then $z'\in V$. Indeed, since $(g,h)$ is a FI-morphism, from (N2) of Definition \ref{Definition FI-morphism}, there exist $x\in Y$ and $z\in X$ such that $(x,y,z)\in T$, $x'\leq_{\hat{Y}} h(x)$ and $g(z)\leq_{\hat{X}} z'$. As $U\in \mathcal{P}_{i}(Y)$, it follows that $h(x)\in U$ and since $(x,y,z)\in T$ from assumption, then we obtain $g(z)\in V$ and $z' \in V$. The remaining inclusion is easy. 
\end{proof}

Observe that Lemmas \ref{representation} and \ref{Representation of FI-frame morphisms} allow to define a functor $\mathbb{G} \colon \mathsf{FIFram} \to \mathsf{FIMod}^{op}$ as follows: 
\begin{displaymath}
\begin{array}{rcl}
\mathcal{F} & \mapsto & \mathcal{M}_{\mathcal{F}}
\\
(g,h) & \mapsto & (g^{\ast},h^{\ast}).
\end{array}
\end{displaymath}

On the other hand, from Lemmas \ref{Representation of a FI-frame} and \ref{Representation of FIDL morphisms}, we can define a functor $\mathbb{F} \colon \mathsf{FIMod}^{op} \to \mathsf{FIFram}$ as follows:
\begin{displaymath}
\begin{array}{rcl}
\mathcal{M} & \mapsto & \mathcal{F}_{\mathcal{M}}
\\
(\alpha,\gamma) & \mapsto & (\alpha^{\ast},\gamma^{\ast}).
\end{array}
\end{displaymath}

We conclude this section by proving our representation theorem for FIDL-modules. 

\begin{theorem} \label{Functor 2 DLFI}
$\mathbb{G}$ is a left adjoint of $\mathbb{F}$ and the counit is an isomorphism.
\end{theorem}

\begin{proof}
We start by showing that $\mathbb{G}$ is a left adjoint of $\mathbb{F}$. Let $\mathcal{F}=\langle X,Y,\leq_{X},\leq_{Y}, R,T\rangle$ be a FI-frame, $\mathcal{M}=\langle {\bf{A}}, {\bf{B}}, f, i \rangle$ a FIDL-module and $(g,h) \colon \mathcal{F} \to \mathcal{F}_{\mathcal{M}}$ a FI-morphism. Then $g \colon X \to \mathcal{X}(\textbf{A})$ and $h \colon Y \to \mathcal{X}(\textbf{B})$ are maps of posets satisfying the conditions of Definition \ref{Definition FI-morphism}. From Stone's representation theorem, there exist a unique pair of lattice homomorphisms $\overline{g} \colon A \to \mathcal{P}_{i}(X)$ and $\overline{h} \colon B \to \mathcal{P}_{i}(Y)$ defined by $\overline{g}(a) = \{y \in X \colon a \in g(y)\}$ and $\overline{h}(b) = \{x \in Y \colon b \in h(x)\}$. In order to prove that $(\overline{g}, \overline{h}) \colon \mathcal{M} \to \mathcal{M}_{\mathcal{F}}$ is a FIDL-homomorphism, we need to show the commutativity of the following diagrams
\begin{displaymath}
\begin{array}{ccc}
\xymatrix{
A \ar[r]^-{f_{b}} \ar[d]_-{\overline{g}}  & A \ar[d]^-{\overline{g}}
\\
\mathcal{P}_{i}(X) \ar[r]_-{f_{{\mathcal{F}}_{\overline{h}(b)}}} & \mathcal{P}_{i}(X)
} & & \xymatrix{
A \ar[r]^-{i_{b}} \ar[d]_-{\overline{g}}  & A \ar[d]^-{\overline{g}}
\\
\mathcal{P}_{i}(X) \ar[r]_-{i_{{\mathcal{F}}_{\overline{h}(b)}}} & \mathcal{P}_{i}(X)
}
\end{array}
\end{displaymath}  
for every $b \in B$. We prove $\overline{g}(f_{b}(a)) = f_{\mathcal{F}_{\overline{h}(b)}}(\overline{g}(a))$, for every $a \in A$. So to check that $\overline{g}(f_{b}(a)) \subseteq f_{\mathcal{F}_{\overline{h}(b)}}(\overline{g}(a))$, let $z\in \overline{g}(f_{b}(a))$. Thus, $f_{b}(a)\in g(z)$. By Lemma \ref{lem_1}, there exist $Q \in \mathcal{X}(\textbf{B})$ and $E \in \mathcal{X}(\textbf{A})$ such that $b \in Q$, $a \in E$ and $(Q,E, g(z)) \in R_{\mathcal{M}}$. From condition (M2), there exist $x\in Y$ and $y\in X$ such that $Q\subseteq_{\mathcal{X}(\textbf{B})} h(x)$, $E \subseteq_{\mathcal{X}(\textbf{A})} g(x)$ and $(x,y,z) \in R$. Hence, $x\in \overline{h}(b)$ and $y \in \overline{g}(a)$. Therefore $z \in  f_{\mathcal{F}_{\overline{h}(b)}}(\overline{g}(a))$. Conversely, if $z \in f_{\mathcal{F}_{\overline{h}(b)}}(\overline{g}(a))$, then there exist $x \in \overline{h}(b)$ and $y \in \overline{g}(a)$ such that $(x,y,z)\in R$. So, $(h(x),g(y),g(z)) \in R_{\mathcal{M}}$ from (M1) and consequently, $f(h(x),g(y))\subseteq g(z)$. By Lemma \ref{lem_1}, $f_{b}(a) \in g(z)$, or equivalently, $z \in \overline{g}(f_{b}(a))$.

Now we prove that $\overline{g}(i_{b}(a)) = i_{\mathcal{F}_{\overline{h}(b)}}(\overline{g}(a))$, for every $a \in A$. Let $z \in \overline{g}(i_{b}(a))$. If $(x,y,z)\in T$ and $a \in h(x)$, then from condition (N1) we have $i(h(x),g(y)) \subseteq_{\mathcal{X}(\textbf{A})} g(z)$. Since $i_{b}(a) \in g(z)$, from Lemma \ref{lem_1}, we can conclude that $a \in g(z)$ and $z \in i_{\mathcal{F}_{\overline{h}(b)}}(\overline{g}(a))$. On the other hand, assume that $z \in i_{\mathcal{F}_{\overline{h}(b)}}(\overline{g}(a))$. In order to prove that $z\in \overline{g}(i_{b}(a))$, let $R\in \mathcal{X}(\textbf{B})$ and $Q \in \mathcal{X}(\textbf{A})$ such that $(R,g(y),Q)\in T_{\mathcal{M}}$ and $b \in R$. So, $i(R,g(y))\subseteq_{\mathcal{X}(\textbf{A})} Q$. From condition (N2), then there exist $x \in Y$ and $z\in X$ such that $(x,y,z)\in T$, $R\subseteq_{\mathcal{X}(\textbf{B})}h(x)$ and $g(z)\subseteq_{\mathcal{X}(\textbf{B})}Q$. So, from assumption we have $a \in g(z)$. Hence, by Lemma \ref{lem_1}, $z \in \overline{g}(i_{b}(a))$. 

For the last part, note that for each frame $\mathcal{F}=\langle X,Y,\leq_{X},\leq_{Y}, R,T\rangle$, the counit of the adjunction $\mathbb{G} \dashv \mathbb{F}$ is determined by the pair of monotone maps $\epsilon_{X} \colon X \to X(\mathcal{P}_{i}(X))$ and $\epsilon_{Y} \colon Y \to X(\mathcal{P}_{i}(Y))$, which are defined by $\epsilon_{X}(y) = \{V \in \mathcal{P}_{i}(X) \colon y\in V\}$ and $\epsilon_{Y}(x) = \{U \in \mathcal{P}_{i}(Y) \colon x\in U\}$. It is clear from Stone's representation theorem, that the maps $\epsilon_{X}$ and $\epsilon_{Y}$ are isomorphisms of posets. So, from Lemma \ref{Isos FIF}, the result follows.
\end{proof}

\section{Topological duality}   \label{Topological duality for FIDL-modules}

In this section we prove a duality for FIDL-modules by using some of the results of Section \ref{Representation of FIDL-modules} together with a suitable extension of Priestley duality for distributive lattices. The dual objects are certain topological bi-spaces endowed with two relations satisfying some particular properties. 

\begin{definition}\label{Associated Spaces}
A structure $\mathcal{U}=\langle X,Y,\leq_{X}, \leq_{Y}, \tau_{X}, \tau_{Y}, R, T\rangle$ is called an {\rm{Urquhart space}}, if the following conditions hold:
\begin{enumerate}
\item $\langle X, \leq_{X}, \tau_{X} \rangle$ and $\langle Y, \leq_{Y}, \tau_{Y} \rangle$ are Priestley spaces,
\item $R\subseteq X\times Y\times X$ and $T\subseteq Y\times X \times X$,
\item For every $U\in \mathcal{C}(Y)$ and every $V \in \mathcal{C}(X)$, we have $f(V,U), i(U,V)\in \mathcal{C}(X)$,
\item For every $x \in Y$ and every $y,z\in X$, if $f(\epsilon_{X}(y) ,\epsilon_{Y}(x))\subseteq \epsilon_{X}(z)$, then $(y, x, z)\in R$,
\item For every $x\in Y$ and every $y,z\in X$, if $i(\epsilon_{Y}(x) ,\epsilon_{X}(y))\subseteq \epsilon_{X}(z)$, then $(x,y,z)\in T$.
\end{enumerate}
\end{definition}

\begin{lemma}\label{Spaces are Frames}
If $\mathcal{U}$ is an Urquhart spaces, then it is a FI-frame.
\end{lemma}
\begin{proof}
Let $\mathcal{U}=\langle X,Y,\leq_{X}, \leq_{Y}, R,T\rangle$ be an Urquhart space. We need to check that $R$ and $T$ satisfy conditions (\ref{condition_R}) and (\ref{condition_T}) of Definition \ref{definition FI-frames}, respectively. Since both proofs are similar, we only prove that $R$ satisfies (\ref{condition_R}). Suppose $(y,x,z)\in R$, $y'\leq_{X}  y$, $x' \leq_{Y} x$ and $z \leq_{X} z'$. Because of $\epsilon_{X}$ and $\epsilon_{Y}$ are monotonous, then $\epsilon_{X}(y')\subseteq \epsilon_{X}(y)$ and $\epsilon_{Y}(x')\subseteq \epsilon_{Y}(x)$. So, from Proposition \ref{propo_1}, we obtain $f(\epsilon_{X}(y'), \epsilon_{Y}(x'))\subseteq f(\epsilon_{X}(y) ,\epsilon_{Y}(x))$. If $P\in f(\epsilon_{X}(y') ,\epsilon_{Y}(x'))$, then there exist $S\in \mathcal{C}(X)$ and $Q\in \mathcal{C}(Y)$ such that $y\in S$, $x\in Q$ and $f(S,Q)\subseteq P$. As $\mathcal{U}$ is an Urquhart space then $f(S,Q)\in \mathcal{C}(X)$, and due to $(y,x,z)\in R$ from hypothesis, we obtain $z\in f(S,Q)$. Thus, $z\in P$. Because $P\in \mathcal{C}(X)$ and $z\leq_{X}z'$, we conclude that $f(\epsilon_{Y}(x'),\epsilon_{X}(y'))\subseteq \epsilon_{X}(z')$. Therefore, from $(4)$ of Definition \ref{Associated Spaces}, we get $(y',x',z')\in R$.
\end{proof}

\begin{definition}
Let $\mathcal{U}=\langle X,Y,\leq_{X}, \leq_{Y}, \tau_{X}, \tau_{Y}, R, T\rangle$ and $\hat{\mathcal{U}}=\langle \hat{X},\hat{Y},\leq_{\hat{X}}, \leq_{\hat{Y}}, \tau_{\hat{X}}, \tau_{\hat{X}} \hat{R}, \hat{T}\rangle$ be Urquhart spaces. We shall say that a pair $(g,h) \colon \mathcal{U} \to \hat{\mathcal{U}}$ is an {\rm{U-map}}, if $g \colon X \to \hat{X}$ and $h \colon Y \to \hat{Y}$ are monotonous and continuous maps, and satisfy the conditions (M1), (M2), (N1) and (N2) of Definition \ref{Definition FI-morphism}. 
\end{definition}

We denote by $\mathsf{USp}$ the category of Urquhart spaces and U-maps.

Let $\mathcal{M}$ be a FIDL-module and $\mathbb{F} \colon \mathsf{FIMod}^{op} \to \mathsf{FIFram}$ the functor of Theorem \ref{Functor 2 DLFI}. Notice that $\mathbb{F}(\mathcal{M})=\mathcal{F}_{\mathcal{M}}$ is an Urquhart space. The latter assertion lies in the following facts which are immediate from Priestley duality: (1) $\langle\mathcal{X}(\textbf{A}),\subseteq_{\bf{A}}, \tau_{\textbf{A}}\rangle$ and $\langle\mathcal{X}(\textbf{B}), \subseteq_{\bf{B}}, \tau_{\textbf{B}}\rangle$ are Priestley spaces; (2) Since $\mathcal{C}(\mathcal{X}(\textbf{A}))=\{\beta_{\bf{A}}(a) \colon  a\in A \}$ and $\mathcal{C}(\mathcal{X}(\textbf{B}))=\{\beta_{\bf{B}}(b) \colon b\in B\}$ then, for every $U\in \mathcal{C}(\mathcal{X}(\textbf{B}))$ and every $V \in \mathcal{C}(\mathcal{X}(\textbf{A}))$, there exist $a\in A$ and $b\in B$ such that $f(V,U)=f(\beta_{\bf{A}}(a), \beta_{\bf{B}}(b))=\beta_{\bf{A}}(f(a,b))$ and $i(U,V)=i(\beta_{\bf{B}}(b),\beta_{\bf{A}}(a))=\beta_{\bf{A}}(i(b,a))$. Hence $f(V,U), i(U,V)\in \mathcal{C}(\mathcal{X}(\textbf{A}))$; (3) Since $\epsilon_{\mathcal{X}(\textbf{A})}(P)=P$ and $\epsilon_{\mathcal{X}(\textbf{B})}(Q)=Q$, for every $P\in \mathcal{X}(\textbf{A})$ and every $Q\in \mathcal{X}(\textbf{B})$, it follows that conditions $(4)$ and $(5)$ of the Definition \ref{Associated Spaces} hold. In addition, if $(\alpha,\gamma) \colon \mathcal{M} \to \hat{\mathcal{M}}$ is a FIDL-homomorphism between two FIDL-modules $\mathcal{M}$ and $\hat{\mathcal{M}}$, it is also clear from Priestley duality that $\mathbb{F}(\alpha, \gamma) = (\alpha^{\ast}, \lambda^{\ast})$ is an U-map between Urquhart spaces.

On the other hand, if $\mathcal{U}$ is an Urquhart space, then from Priestley duality the structure
\begin{equation*}
\mathcal{M}_{\mathcal{U}}= \langle\mathcal{C}(X), \mathcal{C}(Y),f_{\mathcal{U}},i_{\mathcal{U}}\rangle
\end{equation*}
is a FIDL-module. These facts allows us to define an assignment $\mathbb{J} \colon \mathsf{USp} \rightarrow \mathsf{FIMod}^{op}$ as follows: 
\begin{displaymath}
\begin{array}{rcl}
\mathcal{U} & \mapsto & \mathcal{M}_{\mathcal{U}}
\\
(g,h) & \mapsto & (g^{\ast},h^{\ast}),
\end{array}
\end{displaymath}
where $f_{\mathcal{U}}$ and $i_{\mathcal{U}}$ are the operations defined in (\ref{f_in P}) and (\ref{i_in P}), respectively. Such an assignment is clearly functorial. Notice that as an straight application of Priestley duality, it follows that $\mathbb{J}$ is the inverse functor of $\mathbb{F}$. Since this is routine, we leave to the reader the details of the proof of the following result. 

\begin{theorem} \label{Duality for DLFI-modules}
The categories $\mathsf{FIMod}$ and $\mathsf{USp}$ are dually equivalent.
\end{theorem}

\section{Congruences of FIDL-modules}  \label{Congruences of FIDL-modules}

In this section we introduce the concept of congruence in the class of FIDL-modules and we show how through the duality of Section \ref{Topological duality for FIDL-modules} we can provide a characterization of these in terms of certain pairs of closed subsets of the associated Urquhart space. This result will allows us to give a topological bi-spaced characterization of the simple and subdirectly irreducible FIDL-modules.  

\begin{definition} \label{DLFI Congruences}
Let $\mathcal{M}$ be a FIDL-module. Let $\theta_{\bf{A}}$ be a congruence of $\bf{A}$ and $\theta_{\bf{B}}$ a congruence of $\bf{B}$. 
\begin{itemize}
\item[(C1)] A pair $\left( \theta_{\bf{A}}, \theta_{\bf{B}} \right) \subseteq A^{2} \times B^{2}$ is called a {\rm{FDL-congruence of $\langle {\bf{A}}, {\bf{B}}, f \rangle$}}, if for every $(a,c) \in \theta_{\bf{A}}$ and every $(b,d) \in \theta_{\bf{B}}$, we have $(f(a,b), f(c,d)) \in \theta_{\bf{A}}$.
\item[(C2)] A pair $\left( \theta_{\bf{A}}, \theta_{\bf{B}} \right) \subseteq A^{2} \times B^{2}$ is called an {\rm{IDL-congruence of $\langle {\bf{A}}, {\bf{B}}, i \rangle$}}, if for every $(a,c) \in \theta_{\bf{A}}$ and every $(b,d) \in \theta_{\bf{B}}$, we have $(i(b,a), i(d,c)) \in \theta_{\bf{A}}$.
\end{itemize}
Moreover, a pair $\left( \theta_{\bf{A}}, \theta_{\bf{B}} \right) \subseteq A^{2} \times B^{2}$ is called a {\rm{FIDL-congruence of $\mathcal{M}$}}, if $\left( \theta_{\bf{A}}, \theta_{\bf{B}} \right)$ is a FDL-congruence and an IDL-congruence.
\end{definition}

If $\mathcal{M}$ is a FIDL-module, then we write $Con_{f}(\mathcal{M})$ for the set of all FDL-congruences, $Con_{i}(\mathcal{M})$ for the set of all IDL-congruences, and $Con(\mathcal{M})$ for the set of all FIDL-congruences of $\mathcal{M}$. It is not hard to see that $Con(\mathcal{M})$ is an algebraic lattice.

We now proceed to introduce the topological notions required for our characterization. Let $\mathcal{U}$ be an Urquhart space. We define the following subsets of $X$ and $Y$:

\begin{itemize}
\item For every $x,z \in X$ and every $y \in Y$, we have
\[R^{1}(y,z) = \{x \in X \colon x {\hspace{0.1cm}} \text{is maximal in $X$ and} {\hspace{0.1cm}} (x,y,z) \in R\},\]
\[R^{2}(x,z) = \{y \in Y \colon y {\hspace{0.1cm}} \text{is maximal in $Y$ and} {\hspace{0.1cm}} (x,y,z) \in R\},\]
\[T^{1}(x,z) = \{y \in Y \colon y {\hspace{0.1cm}} \text{is maximal in $Y$ and} {\hspace{0.1cm}} (y,x,z) \in T\},\]
\[T^{3}(y,x) = \{z \in X \colon z {\hspace{0.1cm}} \text{is minimal in $X$ and} {\hspace{0.1cm}} (y,x,z)\in T\}.\]
\item For every $x,z \in X$, we have 
\[Max(R^{-1}(z)) =\{ (x,y) \in X \times Y \colon x \in R^{1}(y,z) {\hspace{0.1cm}} \text{and} {\hspace{0.1cm}} y \in R^{2}(x,z)\},\]
\[\mathcal{D}(x)  =\{ (y,z) \in Y \times X \colon y \in T^{1}(x,z) {\hspace{0.1cm}} \text{and} {\hspace{0.1cm}} z \in T^{3}(y,x)\}.\]
\end{itemize}

\begin{definition}   \label{strongly closed sets}
Let $\mathcal{U}$ be an Urquhart space. Let $Z_{1}$ be a closed set of $X$ and $Z_{2}$ a closed set of $Y$.
\begin{itemize}
\item[(CL1)] A pair $(Z_{1},Z_{2}) \subseteq X \times Y$ is called a {\rm{$R$-closed set of $\mathcal{U}$}}, if for every $z \in Z_{1}$, we have $Max(R^{-1}(z)) \subseteq Z_{1} \times Z_{2}$.
\item[(CL2)] A pair $(Z_{1},Z_{2}) \subseteq X \times Y$ is called a {\rm{$T$-closed set of $\mathcal{U}$}}, if for every $x \in Z_{1}$, we have $\mathcal{D}(x) \subseteq Z_{2} \times Z_{1}$.
\end{itemize}
Moreover, a pair $(Z_{1},Z_{2}) \subseteq X \times Y$ is called a {\rm{strongly closed set of $\mathcal{U}$}}, if $(Z_{1},Z_{2})$ is both a $R$-closed set and a $T$-closed set. 
\end{definition}

If $\mathcal{U}$ is an Urquhart space, then we write $\mathcal{C}_{f}(\mathcal{U})$ for the set of all $R$-closed sets of $\mathcal{U}$, $\mathcal{C}_{i}(\mathcal{U})$ for the set of all $T$-closed sets of $\mathcal{U}$, and $\mathcal{C}_{s}(\mathcal{U})$ for the set of all strongly closed sets of $\mathcal{U}$.

\begin{theorem}  \label{Characterization of congruences}
Let $\mathcal{M}$ be a FIDL-module and $\mathcal{F}_{\mathcal{M}}$ be the Urquhart space associated of $\mathcal{M}$. We consider the correspondence $(Z_{1},Z_{2}) \to (\theta(Z_{1}),\theta(Z_{2}))$ for every $Z_{1} \in \mathcal{C}( \mathcal{X}(\textbf{A}))$ and every $Z_{2} \in \mathcal{C}( \mathcal{X}(\textbf{B}))$, where $\theta(-)$ is given by (\ref{congruence-closed}). Then:
\begin{enumerate}
\item There exists an anti-isomorphism between $\mathcal{C}_{f}(\mathcal{F}_{\mathcal{M}})$ and $Con_{f}(\mathcal{M})$.
\item There exists an anti-isomorphism between $\mathcal{C}_{i}(\mathcal{F}_{\mathcal{M}})$ and $Con_{i}(\mathcal{M})$.
\item There exists an anti-isomorphism between $\mathcal{C}_{s}(\mathcal{F}_{\mathcal{M}})$ and $Con(\mathcal{M})$.
\end{enumerate}
\end{theorem}

\begin{proof}
Since $(3)$ is clearly a straight consequence of $(1)$ and $(2)$, we only prove such items.

$(1)$ Let us assume that $(Z_{1},Z_{2})$ is a $R_{\mathcal{M}}$-closed set of $\mathcal{F}_{\mathcal{M}}$. We prove that $(\theta(Z_{1}),\theta(Z_{2}))$ is a FDL-congruence. Let $(x,y) \in \theta(Z_{1})$ and $(b,c) \in \theta(Z_{2})$. If $P \in \beta_{\bf{A}}(f(x,b)) \cap Z_{1}$, then $f(x,b) \in P$. By Lemma \ref{lem_1}, there exist $Q \in \mathcal{X}(\textbf{A})$ and $R \in \mathcal{X}(\textbf{B})$ such that $f(Q,R) \subseteq P$, $x \in Q$ and $b \in R$. Using Zorn's Lemma, it is easy to prove that there are $Q' \in \mathcal{X}(\textbf{A})$ and $R' \in \mathcal{X}(\textbf{B})$ maximals such that $(Q',R') \in Max(R_{\mathcal{M}}^{-1}(P))$. Since $(Z_{1},Z_{2})$ is a $R_{\mathcal{M}}$-closed set of $\mathcal{F}_{\mathcal{M}}$ by assumption, then it follows that $Q' \in Z_{1}$ and $R' \in Z_{2}$. Thus, $f(Q',R') \subseteq P$, $y \in Q'$ and $c \in R'$. Then, by Lemma \ref{lem_1}, we have $f(y,c) \in P$. So, $P \in \beta_{\bf{A}}(f(y,c)) \cap Z_{1}$. The other inclusion is similar. Therefore $(f(x,b), f(y,c)) \in \theta(Z_{1})$.

For the converse, let $(\theta(Z_{1}),\theta(Z_{2}))$ be a FDL-congruence and suppose that the pair $(Z_{1},Z_{2})$ is not a $R_{\mathcal{M}}$-closed set of $\mathcal{F}_{\mathcal{M}}$. Then, there exist $P \in Z_{1}$ and $(Q,R) \in \mathcal{X}(\textbf{A}) \times \mathcal{X}(\textbf{B})$ such that $(Q,R) \in Max(R_{\mathcal{M}}^{-1}(P))$ and $(Q,R) \notin Z_{1} \times Z_{2}$. Suppose that $Q \notin Z_{1}$. Since $Z_{1}$ is a closed set of $\mathcal{X}(\textbf{A})$, then there exist $a,b \in A$ such that $a \in Q$, $b \notin Q$ and $(a \wedge b,a) \in \theta (Z_{1})$. Let us consider the filter ${\rm{Fig}}_{\textbf{A}}( Q \cup \{b\})$. As $Q \in R^{1}_{\mathcal{M}}(R,P)$, then $Q$ is maximal and $f({\rm{Fig}}_{\textbf{A}}(Q \cup \{b\}), R) \nsubseteq P$. So, there exist $q \in Q$ and $r \in R$ such that $f(q \wedge b, r) \notin P$. Since $(\theta(Z_{1}),\theta(Z_{2}))$ is a FDL-congruence, it follows that $(f(a \wedge b \wedge q, r), f(a \wedge q, r)) \in \theta(Z_{1})$. Now, since $a \wedge q \in Q$, then $f(a \wedge q, r) \in f(Q,R) \subseteq P$. Hence $f(a \wedge b \wedge q, r) \in P$. Notice that $f(a \wedge b \wedge q, r) \leq f(b \wedge q, r )$, therefore $f(b \wedge q, r) \in P$, which is a contradiction. Then $Q \in Z_{1}$. The proof of $R \in Z_{2}$ is similar. So, $(Z_{1},Z_{2})$ is a $R_{\mathcal{M}}$-closed.

$(2)$ Assume that $(Z_{1},Z_{2})$ is a $T_{\mathcal{M}}$-closed set of $\mathcal{F}_{\mathcal{M}}$. Let $(x,y) \in \theta(Z_{1})$ and $(b,c) \in \theta(Z_{2})$. We will see that $(\theta(Z_{1}), \theta(Z_{2}))$ is an IDL-congruence of ${\bf{A}}$. Let $P \in \mathcal{X}(\textbf{A})$. Suppose that $P \in \beta_{\bf{A}}(i(b,x)) \cap Z_{1}$ and $P \notin \beta_{\bf{A}}(i(c,y)) \cap Z_{1}$, i.e., $i(b,x) \in P$ and $i(c,y) \notin P$. By Lemma \ref{lem_1}, there exist $R \in \mathcal{X}(\textbf{B})$ and $Q \in \mathcal{X}(\textbf{A})$ such that $i(R,P) \subseteq Q$, $c \in R$ and $y \notin Q$. Note that from Zorn's Lemma it is not hard to see that there are $R' \in \mathcal{X}(\textbf{B})$ and $Q' \in \mathcal{X}(\textbf{A})$ such that $(R',Q') \in \mathcal{D}(P)$. Since $(Z_{1},Z_{2})$ is a $T_{\mathcal{M}}$-closed set, then $R' \in Z_{2}$ and $Q' \in Z_{1}$. Due to $R \subseteq R'$, we have $c \in R'$ and because $(b,c) \in \theta(Z_{2})$, it follows that $b \in R'$. On the other hand, since $i(b,x) \in P$, $i(R',P) \subseteq Q'$ and $b \in R'$, by Lemma \ref{lem_1} we have $x \in Q'$. Then $(x,y) \in \theta(Z_{1})$, $y \in Q' \subseteq Q$ and $y \in Q$, which is a contradiction. We conclude that $(\theta(Z_{1}), \theta(Z_{2}))$ is an IDL-congruence.

Conversely, we assume $(\theta(Z_{1}), \theta(Z_{2}))$ is an IDL-congruence. Suppose that $(Z_{1},Z_{2})$ is not a $T_{\mathcal{M}}$-closed set of $\mathcal{F}_{\mathcal{M}}$. Then there exist $P \in Z_{1}$, $Q \in \mathcal{X}(\textbf{A})$ and $R \in \mathcal{X}(\textbf{B})$ such that $(R,Q) \in \mathcal{D}(P)$ and $(R,Q) \notin Z_{2} \times Z_{1}$. If $R \notin Z_{2}$, then since $Z_{2}$ is a closed set of $\mathcal{X}(\textbf{B})$ there exist $b,c \in B$ such that $b \in R$, $c \notin R$ and $(b \wedge c, b) \in \theta(Z_{2})$. Let us consider ${\rm{Fig}}_{\textbf{B}}(R \cup \{c\})$. Since $R \in T_{\mathcal{M}}^{1}(P,Q)$, then $i({\rm{Fig}}_{\textbf{B}}(R \cup \{c\}), P) \nsubseteq Q$, i.e., there exists $x \in A$ such that $x \notin Q$ and $p \leq i(r \wedge c, z)$, for some $r \in R$ and $p \in P$. Hence $i(r \wedge c, z) \in P$ and by Proposition \ref{propo_1}, $i(r \wedge b \wedge c, z) \in P$. On the other hand, since $(\theta(Z_{1}), \theta(Z_{2}))$ is a congruence, we obtain that $(i(r \wedge b \wedge c, z), i(r \wedge b, z)) \in \theta(Z_{1})$ and $i(r \wedge b, z) \in P$. So, $i(R,P) \subseteq Q$, $r \wedge b \in R$ and by Lemma \ref{lem_1} we have $z \in Q$, which is a contradiction. If $Q \notin Z_{1}$, then there exist $x,y \in A$ such that $x \in Q$, $y \notin Q$ and $(x \wedge y, x) \in \theta(Z_{1})$. Let us consider $I={\rm{Idg}}_{\textbf{A}}(Q^{c} \cup \{y\})$. Observe that $I \cap i(R,P) \neq \emptyset$, because otherwise from the Prime Filter Theorem, there would exists $H \in \mathcal{X}(\textbf{A})$ such that $i(R,P) \subseteq H$, $H \subseteq Q$ and $x \notin H$ which is absurd since $Q$ is minimal. Thus, there exist $a \in A$ such that $a \leq q \vee x$ and $p \leq i(r,a)$, for some $q \notin Q$, $r \in R$ and $p \in P$. So, by Proposition \ref{propo_1}, $p \leq i(r,a) \leq i(r, q \vee x)$ and $i(p, q \vee x) \in P$. Therefore, since $(\theta(Z_{1}), \theta(Z_{2}))$ is a congruence, it follows that $(i(r, q \vee (x \wedge y)), i(r, q \vee x)) \in \theta(Z_{1})$. Hence $i(r, q \vee (x \wedge y)) \in P$. Since $i(R,P) \subseteq Q$ and $r \in R$, then by Lemma \ref{lem_1} we get that $q \vee (x \wedge y)  \in Q$, which is a contradiction because $Q$ is prime. Then $(Z_{1},Z_{2})$ is a $T_{\mathcal{M}}$-closed set.
\end{proof}

Let $\{\mathcal{M}_{k}\}_{k\in K}$ be a family of FIDL-modules, with $\mathcal{M}_{k}=\langle \textbf{A}_{k}, \textbf{B}_{k},f_{k},i_{k}\rangle$. Then 
\begin{equation*}
\underset{k\in K}{\prod}\mathcal{M}_{k} = \left\langle \underset{k\in K}{\prod}{\bf{A}_{k}}, \underset{k\in K}{\prod}{\bf{B}_{k}}, f, i \right\rangle
\end{equation*}
has a FIDL-module structure, where $f(a,b)(k)=f_{k}(a(k),b(k))$ and $i(b,a)(k)=i_{k}(b(k),a(k))$, for every $k \in K$. Let $\pi^{\bf{A}}_{k}: \underset{k\in K}{\prod}{\bf{A}_{k}} \rightarrow \bf{A}_{k}$ and $\pi^{\bf{B}}_{k}:\underset{k\in K}{\prod}{\bf{B}_{k}} \rightarrow \bf{B}_{k}$ be the projection homomorphisms. Note that the pair $(\pi^{\bf{A}}_{k}, \pi^{\bf{B}}_{k})$ is a FIDL-homomorphism, for every $k \in K$. It is no hard to see that $\underset{k\in K}{\prod}\mathcal{M}_{k}$ together with the family $\{(\pi^{\bf{A}}_{k}, \pi^{\bf{B}}_{k})\}_{k \in K}$ is in fact the categorical product of $\{\mathcal{M}_{k}\}_{k\in K}$.

Let $(\alpha,\gamma)$ be a FIDL-homomorphism. We say that $(\alpha,\gamma)$ is a 1-1 FIDL-homomorphism if $\alpha$ and $\gamma$ are 1-1, and similarly, we say that $(\alpha,\gamma)$ is a onto FIDL-homomorphism if $\alpha$ and $\gamma$ are onto.

If $\mathcal{M}$ is a FIDL-module, then we introduce the following concepts: 
\begin{itemize}
\item We will say that $\mathcal{M}$ is a {\it{subdirect product}} of a family $\{\mathcal{M}_{k}\}_{k\in K}$ of FIDL-modules, if there exists a 1-1 FIDL-homomorphism 
\begin{equation*}
(\alpha, \gamma) \colon \mathcal{M} \to \underset{k\in K}{\prod}\mathcal{M}_{k}
\end{equation*}
such that $(\pi^{\bf{A}}_{k} \alpha, \pi^{\bf{B}}_{k} \gamma)$ is an onto FIDL-homomorphism, for every $k\in K$.
\item We will say that $\mathcal{M}$ is \emph{subdirectly irreducible} if for every family of FIDL-modules $\{\mathcal{M}_{k}\}_{k\in K}$ and 1-1 FIDL-homomorphism 
\begin{equation*}
(\alpha, \gamma) \colon \mathcal{M} \to \underset{k\in K}{\prod}\mathcal{M}_{k}
\end{equation*}
there exists a $k\in K$ such that $(\pi^{\bf{A}}_{k} \alpha, \pi^{\bf{B}}_{k} \gamma)$ is an isomorphism of FIDL-modules.
\item We will say that $\mathcal{M}$ is \emph{simple} if the lattice of the FIDL-congruences has only two elements.
\end{itemize}

The following result is immediate from Theorem \ref{Characterization of congruences}.

\begin{corol}  \label{subdirectly irreducible DLFI-modules}
Let $\mathcal{M}$ be a FIDL-module. Then $\mathcal{M}$ is subdirectly irreducible if and only if $\mathcal{M}$ is trivial or there exists a minimal non-trivial FIDL-congruence in $\mathcal{M}$.
\end{corol}

If $\mathcal{U}$ is an Urquhart space, then from Theorem \ref{Characterization of congruences} it is clear that $\mathcal{C}_{s}(\mathcal{U})$ is an algebraic lattice. So, if $Z_{1} \times  Z_{2} \subseteq X \times Y$, let ${\rm{cl}}_{\mathcal{C}_{s}}(Z_{1},Z_{2})$ be the smallest element of $\mathcal{C}_{s}(\mathcal{U})$ which contains $Z_{1} \times Z_{2}$. Let $(x,y) \in X \times Y$. If there is no place to confusion, we write ${\rm{cl}}_{\mathcal{C}_{s}}(x,y)$ instead of ${\rm{cl}}_{\mathcal{C}_{s}}(\{x\},\{y\})$.

\begin{prop} \label{Simple algebras}
Let $\mathcal{M}$ be a FIDL-module and $\mathcal{F}_{\mathcal{M}}$ be the Urquhart space associated of $\mathcal{M}$. Then $\mathcal{M}$ is simple if and only if ${\rm{cl}}_{\mathcal{C}_{s}}(P,Q) = \mathcal{X}(\textbf{A}) \times  \mathcal{X}(\textbf{B})$, for every $ (P,Q) \in \mathcal{X}(\textbf{A}) \times \mathcal{X}(\textbf{B})$.
\end{prop}

\begin{proof}
Since $\mathcal{M}$ is simple if and only $Con(\mathcal{M})=\{(\Delta^{\mathbf{A}}, \Delta^{\mathbf{B}}), (\nabla^{\mathbf{A}},\nabla^{\mathbf{B}})\}$, then by Theorem \ref{Characterization of congruences} this is equivalent to $\mathcal{C}_{s}(\mathcal{F}_{\mathcal{M}})=\{(\emptyset,\emptyset),(\mathcal{X}(\textbf{A}),\mathcal{X}(\textbf{B}))\}$ and the result follows. 
\end{proof}

\begin{theorem} \label{Subdirectly irreducible algebras}
Let $\mathcal{M}$ be a FIDL-module and $\mathcal{F}_{\mathcal{M}}$ be the Urquhart space associated of $\mathcal{M}$. Then $\mathcal{M}$ is subdirectly irreducible but no simple if and only if the set 
\begin{equation*}
\mathcal{J} = \{ (P,Q) \in \mathcal{X}(\textbf{A}) \times \mathcal{X}(\textbf{B}) \colon {\rm{cl}}_{\mathcal{C}_{s}}(P,Q)=( \mathcal{X}(\textbf{A}), \mathcal{X}(\textbf{B})) \}
\end{equation*}
is a non-empty open set distinct from $(\mathcal{X}(\textbf{A}), \mathcal{X}(\textbf{B}))$.
\end{theorem}

\begin{proof}
Let us assume that $\mathcal{M}$ is subdirectly irreducible. Then $Con(\mathcal{M})-\{(\Delta^{\mathbf{A}},\Delta^{\mathbf{B}})\}$ has a minimum element. From Theorem \ref{Characterization of congruences}, $\mathcal{C}_{s}(\mathcal{F}_{\mathcal{M}})-(\mathcal{X}(\textbf{A}),\mathcal{X}(\textbf{B}))$ has a maximum element. Let $(Z_{1},Z_{2})$ be such an element. Then $Z_{1}$ and $Z_{2}$ are non-empty. We prove that $\mathcal{J}=(Z_{1},Z_{2})-(\mathcal{X}(\textbf{A}),\mathcal{X}(\textbf{B}))$. On the one hand, if $(P,Q)\notin (Z_{1},Z_{2})$, then $(Z_{1},Z_{2}) \subseteq (Z_{1},Z_{2}) \cup {\rm{cl}}_{\mathcal{C}_{s}}(P,Q)$. So it must be that ${\rm{cl}}_{\mathcal{C}_{s}}(P,Q)=( \mathcal{X}(\textbf{A}), \mathcal{X}(\textbf{B}))$, because if it is not the case, then $(Z_{1},Z_{2})$ it would not be the maximum of $\mathcal{C}_{s}(\mathcal{F}_{\mathcal{M}})-(\mathcal{X}(\textbf{A}),\mathcal{X}(\textbf{B}))$, which is a contradiction. On the other hand, if $(P,Q) \in \mathcal{J} \cap (Z_{1},Z_{2})$, then ${\rm{cl}}_{\mathcal{C}_{s}}(P,Q)=( \mathcal{X}(\textbf{A}), \mathcal{X}(\textbf{B}))=(Z_{1},Z_{2})$, which is absurd from assumption. We conclude the proof by noticing that if $\mathcal{J}$ is a non-empty open set distinct from $(\mathcal{X}(\textbf{A}), \mathcal{X}(\textbf{B}))$, then it is easy to see that $\mathcal{J}-(\mathcal{X}(\textbf{A}), \mathcal{X}(\textbf{B}))$ is the maximum of $\mathcal{C}_{f}(\mathcal{F}_{\mathcal{M}})-(\mathcal{X}(\textbf{A}),\mathcal{X}(\textbf{B}))$. Then the result is an immediate consequence of Theorem \ref{Characterization of congruences}.
\end{proof}

\end{document}